\documentclass[sn-basic,Numbered]{sn-jnl}

\usepackage{graphicx}%
\usepackage{multirow}%
\usepackage{amsmath,amssymb,amsfonts}%
\usepackage{amsthm}%
\usepackage{mathrsfs}%
\usepackage[title]{appendix}%
\usepackage{xcolor}%
\usepackage{textcomp}%
\usepackage{manyfoot}%
\usepackage{booktabs}%
\usepackage{algorithm}%
\usepackage{algorithmicx}%
\usepackage{algpseudocode}%
\usepackage{listings}%

\theoremstyle{thmstyleone}%
\newtheorem{theorem}{Theorem}[section]
\newtheorem{proposition}[theorem]{Proposition}%

\theoremstyle{thmstyletwo}%
\newtheorem{remark}[theorem]{Remark}%

\theoremstyle{thmstylethree}%

\usepackage{psfrag,enumitem}
\usepackage{mathscinet}
\newtheorem{lemma}[theorem]{Lemma}
\newtheorem{corollary}[theorem]{Corollary}

\newcommand{\R}{{\mathbb R}}
\newcommand{\N}{{\mathbb N}}
\providecommand{\norm}[1]{\| #1 \|}
\providecommand{\cof}[1]{\mathrm{Cof}( #1)}
\renewcommand{\div}{\mathrm{div}}
\newcommand{\tr}{\mathrm{tr}\;}
\newcommand{\n}{\mathrm{n}}

\allowdisplaybreaks 

\begin{document}

\title[Motion of a fluid-filled elastic solid]{Local well-posedness of the equations governing the motion of a fluid-filled elastic solid}


\author*[1]{\fnm{Giusy} \sur{Mazzone}}\email{giusy.mazzone@queensu.ca}
%
%

\affil*[1]{\orgdiv{Department of Mathematics and Statistics}, \orgname{Queen's University}, 
\\
\orgaddress{\street{48 University Ave}, \city{Kingston}, \postcode{K7L 3N6}, \state{ON}, \country{Canada}}}
%
%


\abstract{We consider the fluid-structure interaction problem of a viscous incompressible fluid contained in an elastic solid whose motion is not prescribed. The equations governing the motion of the solid are given by the Navier equations of linear elasticity, whereas the fluid motion is described by the Navier-Stokes equations. We prove that the governing equations admit a unique strong solution corresponding to non-zero initial data for the solid initial displacement and velocity, and for a fluid initial velocity in $H^{5/2}$.}

\keywords{fluid-structure interaction problems, well-posedness, hyperbolic-parabolic equations, Navier equations, Navier-Stokes equations}

\pacs[MSC Classification]{35Q30 
\sep 35Q35 
\sep 76D05 
\sep 74B05 
\sep 74B10 
\sep 35K61 
\sep 35L20 
}

\maketitle

\section{Introduction}
We consider the motion of an elastic solid with an interior cavity completely filled with a viscous Newtonian fluid. The equations governing the motion of the  fluid-solid system are given by the coupling of the {\em Navier equations} of linear elasticity for the solid displacement, and the {\em Navier-Stokes equations} for the fluid velocity (see equations \eqref{eq:eom}). At the fluid-solid interface, we impose continuity of the velocity fields and continuity of the tractions. The outer solid boundary is subject to a traction-free boundary condition. We prove that the governing equations, in absence of external forces, admit a unique local (in time) strong solution corresponding to initial data in a low regularity class (see Theorem \ref{th:main}). 

Fluid-solid interactions problems like the one investigated in this paper have been widely studied for their applications in geophysics (\cite{dahlen20}) and in hemodynamics (\cite{quarteroni00}). From the mathematical point of view, the literature on fluid-solid interaction problems is vast. Any attempt to make a comprehensive literature review on this subject would never be completely satisfactory (we refer to \cite{canic21} where some of the aspects mentioned next are treated in some detail). We will mention only works closely related to the topic treated in this paper, and we will leave out many other works, just as important, on fluid-structure interaction problems involving --for example--  compressible fluids. We can classify the works on fluid-solid interaction problems into two broad categories: (a) those pertaining the motion of an elastic solid immersed in a viscous fluid (with fixed outer boundary); and (b) those regarding the motion of a fluid contained in a moving solid (with the motion of the solid not prescribed). For what concerns category (a), first results on the local well-posedness of the equations governing the motion of an elastic solid in a viscous fluid can be found in \cite{coutand05,coutand06} corresponding to initial data in high regularity class, and in \cite{kukavica09,kukavica12,raymond2014,boulakia19} for the case of initial data in low regularity class. Global existence of solutions when the elastic solid is modeled with a damped wave equation, and for small initial data, has been studied in \cite{kukavica14,kukavica17,kukavica23}. We would like to mention also \cite{kukavica18} for related results on similar fluid-structure interactions models, and \cite{BOCIU2020111837} for an optimal control problem aimed at minimizing the fluid vorticity in the case of a damped elastic body moving and deforming within the fluid. Concerning category (b), the majority of the available mathematical results investigates the interactions of a viscous fluid with a structure of lower dimension\footnote{For example, a two-dimensional fluid interacting with a viscoelastic string or a beam, or a three-dimensional fluid interacting with a plate or a Koiter shell. }. Local well-posedness results  can be traced back to the works \cite{daVeiga04,lequeurre11}. The existence of weak solutions has been investigated in \cite{chambolle05,grandmont08,padula08,muha13,lengeler14,muha15}. A global existence result for strong solutions has been proved in \cite{grandmont16}. Existence, stability and long-time behaviour of solutions to the equations governing the motion of a three-dimensional viscous fluid contained in a moving three-dimensional structure, when the structure is a rigid body, has been widely studied by the author of this paper and collaborators in the series of works \cite{DiGaMaZu,MaPrSi,MaPrSi19,gfree}. Existence of solutions for the equations of motions for a fluid-filled elastic solid is widely open. In this paper, we address the problem of existence and uniqueness of local (in time) solutions corresponding to initial data in the regularity class
\[
(u_0,u_1,v_0)\in H^{3}(\Omega_B)\times H^{3/2}(\Omega_B) \times H^{5/2}(\Omega_L), 
\]
where $u_0$ and $u_1$ are the initial displacement and displacement of the solid, respectively, whereas $v_0$ is the initial velocity of the fluid, and they satisfy suitable compatibility conditions (see Theorem \ref{th:main}). In addition, we prove a  results on continuous dependence upon the data (see Corollary \ref{cor:cont-dep-data}).

As for all of the above-mentioned fluid-structure interaction problems, the two main difficulties are given by the unknown motion of the boundary  
and the parabolic-hyperbolic nature of the partial differential equations describing the fluid-structure interaction. As a first step in our treatment, we rewrite the equations of motion in its Lagrangian description, and we then prove that the resulting equations admit a unique local (in time) strong solution. Apart from the solid displacement field and the fluid velocity field, another unknown of our problem is the system's deformation mapping (see equation \eqref{eq:dm}) that maps positions of material points in the chosen reference configuration to the current positions at time $t$. The strategy of the proof of the existence of solutions to the governing equations is based on two fixed point arguments. One in the fluid velocity, for a fixed deformation mapping, that involves solving a non-homogeneous Stokes problem subject to a prescribed condition on the Cauchy stress tensor on the fluid boundary (see \eqref{eq:fluid}), and the Navier equation for elasticity subject to a prescribed condition on the solid velocity on the inner boundary (at the interface with the fluid) and a traction-free condition on the outer boundary (see \eqref{eq:elastic}). The second fixed point argument is on the deformation mapping obtained by solving the flow map problem \eqref{eq:flow}. A similar strategy has been employed in \cite{boulakia19}. However, we wish to remark some important differences with the problem solved in \cite{boulakia19}. In the present paper, we are dealing with a different physical problem. In fact, for the problem at hand, no part of the fluid boundary is fixed and the solid is subject to mixed (and generalized) Neumann- and Dirichlet-type boundary conditions on different portions of the solid boundary. Another difference with \cite{boulakia19} is that we deal with an elastic solid which is {\em initially pre-stressed}, so that the initial displacement of the solid is non-zero (technically speaking, for our problem formulation, we use a ``generic'' reference configuration that does not coincide with the solid configuration at time $t=0$, see  Section \ref{sec:formulation}). Because of the compatibility conditions, the latter assumption requires an initial fluid velocity which is $1/2$ more regular than that assumed in \cite{boulakia19} (which, however, assumes that $u_0=0$). In turn, we obtain more regular solution. We wish to emphasize that, although we did not consider any external force applied to the fluid-solid system, the proof of existence of solutions would follow verbatim if the body forces are in a suitable regularity class (see Remark \ref{rm:forces}). 

We conclude this introduction with an outline of our paper. In Section \ref{sec:preliminaries}, we present some basic notation and useful inequalities that will be employed throughout the text.

In Section \ref{sec:formulation}, we present the equations of motion in their formulation from continuum mechanics, using the Lagrangian description of motion for the solid and the Eulerian description of motion for the fluid (see \eqref{eq:eom}). We next rewrite these equations using only Lagrangian variables (see \eqref{eq:eom-L}). 

In Section \ref{sec:main-th}, we state the main theorem of this paper (Theorem \ref{th:main}) and we outline the strategy of its proof with the two fixed point arguments mentioned above. 

Sections \ref{sec:proof-main0} and \ref{sec:proof-main} contain the proof of local well-posedness of the equations for a fixed deformation mapping (see equations \eqref{eq:eom-L-c-hat}-\eqref{eq:ic-hat}) and the proof of Theorem \ref{th:main} plus few concluding remarks. 

In Appendix \ref{ap:stokes-navier}, we revisit some classical results on the non-homogeneous Stokes problem (see Theorems \ref{th:nStokes} and \ref{th:nStokes_r}), and some regularity results on the Navier equations of elasticity (see Theorem \ref{th:N}). Notably, this latter result is important because it exploits the ``hidden regularity'' shown for similar types of hyperbolic equations in \cite{LLT}. This hidden regularity is just what is needed to make sense of the continuity of the stresses at the fluid-solid interface for the coupled problem. 

\section{Notation and useful inequalities}\label{sec:preliminaries}
We use Einstein summation convention for repeated indices. 

Let $D$ be a domain in $\R^d$, $d\in \N\setminus\{0\}$, and $p\in [1,\infty]$. We denote with $L^p(D)$ and $W^{k,p}(D)$ the Lebesgue and Sobolev spaces endowed with the norms 
\[
\norm{w}_{L^p(D)}:=\left(\int_D |w|^p\;d x\right)^{1/p}\quad\text{and}\quad 
\norm{w}_{W^{k,p}(D)}:=\left(\sum_{|\alpha|\le k}\norm{D^\alpha w}^p_{L^p(D)}\right)^{1/p},
\]
respectively. In the above, we have denoted by $|\cdot|$ the Euclidean norm of vector fields (we will use the same symbol for the Frobenius norm of tensor fields). Also, $D^\alpha$ denotes the $\alpha$-th order weak derivative of $w$. If $k\notin \N$, then $W^{k,p}(D):=B^k_{pp}(D)$. We recall the following characterization of Besov spaces $B^s_{qp}(D)=(H^{s_0}_q(D),H^{s_1}_q(D))_{\theta,p}$ as real interpolation of Bessel potential spaces, and of Bessel potential spaces $H^s_q(D)=[H^{s_0}_q(D),H^{s_1}_q(D)]_{\theta}$. These characterizations are valid for $s_0\ne s_1\in \R$, $p,q\in [1, \infty)$, $\theta\in (0,1)$ and $s=(1-\theta)s_0+\theta s_1$. We also recall that $W^{s,2}(D) =B^s_{22}(D) = H^s_2(D)=:H^s(D)$ (see \cite{bergh}). For every $s>1/2$, $H^s_0(D):=\{f\in H^s(D):\;f=0\text{ on }\partial D\}$. 

If $(X, \norm{\cdot}_X)$ is a Banach space and 
$p\in[1,\infty)$, $L^p(0,T;X)$ (resp. $W^{k,p}(0,T;X)$, $k\in \N$) denotes the space of strongly measurable functions $f$ from $[0,T]$ to $X$ for which $\left(\int_0^T \norm{f(t)}^p_X\; d t\right)^{1/p}<\infty$ (resp. $\sum^k_{\ell=0}\left(\int_0^T \norm{\partial^\ell_t f (t)}^p_X\; d t\right)^{1/p}<\infty$). Similarly, $C^k([0,T];X)$ indicates the space of functions which are $k$-times differentiable with values in $X$, and having $\max_{t\in [0,T]}\norm{\partial^\ell_t \cdot}_X < \infty$, for all $\ell = 0,1,...,k$. 

For an open set $A\subset \R^d$ and an interval $I\subset \R$, and for $s\in \R$, we also denote $H^{s,s}(A\times I):=L^2(I;H^s(A))\cap H^s(I;L^2(A))$ endowed with the norm 
\[
\norm{f}_{H^{s,s}(A\times I)}:=\norm{f}_{L^2(I;H^s(A))}+\norm{f}_{H^s(I;L^2(A))}. 
\] 

In our proofs, we may silently use the following two results (whose proof can be found in \cite[Appendix A]{boulakia19}). 

\begin{lemma}\label{lem:interpolation}
Let $m_1$, $m_2\in [0,\infty)$, $A\subset\R^3$ be a bounded smooth domain and $T\in (0,1]$. \begin{itemize}
\item[(I1)] If $f\in L^2(0,T;H^{m_1}(A))\cap H^1(0,T;H^{m_2}(A))$, then $f\in H^\theta(0,T;H^m(A))$ for all $\theta\in [0,1]$ and $m=(1-\theta)m_1+\theta m_2$. In addition, there is a constant $C_1>0$, independent of $T$, such that 
\[
\norm{f}_{H^\theta(0,T;H^m(A))}\le C_1\norm{f}_{L^2(0,T;H^{m_1}(A))}^{1-\theta}\norm{f}_{H^1(0,T;H^{m_2}(A))}^\theta.
\]
\item[(I2)] If $f\in H^1(0,T;H^{m_1}(A))\cap H^2(0,T;H^{m_2}(A))$, then $f\in H^\theta(0,T;H^m(A))$ for all $\theta\in [1,2]$ and $m=(2-\theta)m_1+(\theta-1)m_2$. In addition, there is a constant $C_2>0$, independent of $T$, such that 
\[
\norm{f}_{H^\theta(0,T;H^m(A))}\le C_2\norm{f}_{H^1(0,T;H^{m_1}(A))}^{2-\theta}\norm{f}_{H^2(0,T;H^{m_2}(A))}^{\theta-1}.
\]
\end{itemize}
\end{lemma}

\begin{lemma}\label{lem:tsigma-s}
Let $\sigma\in (1/2,1]$, $s\in [0,\sigma]$, $T\in(0,1]$, and $\mathcal X$ a Banach space. Then, there exists a constant $C>0$ independent of $T$ such that 
\[
\norm{f}_{H^s(0,T;\mathcal X)}\le CT^{\sigma -s}\norm{f}_{H^\sigma(0,T;\mathcal X)}\qquad\text{for all }f\in {}_0H^\sigma(0,T;\mathcal X),
\]
where ${}_0H^\sigma(0,T;\mathcal X):=\{f=f(x,t)\in H^\sigma(0,T;\mathcal X):\; f(\cdot,0)=0\}$ .  
\end{lemma}

In addition, we may silently use the following lemma which is a consequence of H\"older's inequality.
\begin{lemma}\label{lem:holder-st}
Let $p,q,r\in [1,+\infty]$, $s_1, s_2, s\in [0,\infty)$, and $\theta\in [0,1]$ satisfying the following conditions
\[
\frac 1r=\frac{\theta}{p}+\frac{1-\theta}{q},\qquad\qquad s=\theta s_1+(1-\theta)s_2.
\]
Then, there exists a constant $C>0$, independent of $T$, such that any $f\in L^p(0,T;H^{s_1}(A))\cap L^q(0,T;H^{s_2}(A))$ also satisfies $f\in L^r(0,T;H^s(A))$ with norm 
\[
\norm{f}_{L^r(0,T;H^s(A))}\le C\norm{f}_{L^p(0,T;H^{s_1}(A))}^\theta\norm{f}_{L^q(0,T;H^{s_2}(A))}^{1-\theta},
\]
where $T\in (0,+\infty)$ and $A\subset\R^3$ is a bounded smooth domain. 
\end{lemma}

Next lemma is a particular case of \cite[Lemma A.7]{boulakia19}. 
\begin{lemma}\label{lem:ba7}
Let and $A\subset\R^3$ be a bounded smooth domain, $s\in[0,1/2]$, $\sigma\in (1/2,1]$, $T\in (0,1]$, and $p,q\in [0,\infty]$ such that 
\[
\frac 1p + \frac 1q = \frac 12.
\]
Then, there exists a constant $C>0$, independent of $T$, such that 
\begin{multline*}
\norm{fg}_{H^s(0,T;L^2(A))}\le CT^{\sigma-s-1/2}\norm{f}_{H^s(0,T;L^p(A))}\norm{g}_{H^\sigma(0,T;L^q(A))}
\\
+\norm{f}_{H^s(0,T;L^p(A))}\norm{g(0)}_{L^q(A)},
\end{multline*}
for all $f\in H^s(0,T;L^p(A))$ and $g\in H^\sigma(0,T;L^q(A))$. 
\end{lemma}

The following proposition is an immediate consequence of 
the fact that, for $s$ satisfying \eqref{eq:trace-s} below, $B^s_{pp}(A)$ is the set of the (time-)traces at $t=0$ of functions in 
\[
L^p(0,\infty;H^{s_1}_p(A)) \cap W^{1,p}(0,\infty; H^{s_0}_p(A))).  
\]
In addition, the following two embeddings 
\[\begin{split}
&L^p(0,\infty;H^{s_1}_p(A)) \cap W^{1,p}(0,\infty; H^{s_0}_p(A)))  \hookrightarrow C_b([0,\infty);B^s_{pp}(A)),
\\
&L^p(0,T;H^{s_1}_p(A)) \cap {_0}W^{1,p}(0,T; H^{s_0}_p(A))) \hookrightarrow {_0}C_b([0,T];B^s_{pp}(A))
\end{split}\]
are valid for $0\le s_0< s_1\in \R$, $p\in (1, \infty)$, and with $s=\frac{s_0}{p}+\left(1-\frac 1p\right)s_1$. For the first claim and for the first embedding, use \cite[Proposition 1.13 \& Corollary 1.14]{MR3753604} (with $X=H^{s_0}_p(A)$ and $Y=H^{s_1}_p(A)$ together with the definition of Bessel potentials given above). For the second one, see \cite[Theorem 4.2]{MR2863860}. Note that the notation ``${_0}W$'' (respectively, ``${_0}C$'') stands for the set of functions for which $u(\cdot,0)=0$ in $A$. For the second embedding, the embedding constant is independent of $T$. 

\begin{proposition}\label{prop:traceat0}
Let $A\subset \R^3$ be a bounded smooth domain and $T\in (0,\infty)$. Suppose that $0\le s_0< s_1\in \R$, $p\in (1, \infty)$, and take 
\begin{equation}\label{eq:trace-s}
s=\frac{s_0}{p}+\left(1-\frac 1p\right)s_1.
\end{equation}
If $w\in L^p(0,T;H^{s_1}_p(A)) \cap W^{1,p}(0,T; H^{s_0}_p(A)))$ and $w(\cdot,0)\in B^s_{pp}(A)$, then 
\[
w\in C([0,T];B^s_{pp}(A)),
\]
and there exists a constant $C>0$, independent of $T$, such that 
\begin{multline*}
\sup_{t\in [0,T]}\norm{w(\cdot,t)}_{B^s_{pp}(A)}\le C\left(\norm{w}_{W^{1,p}(0,T; H^{s_0}_p(A)))\cap L^p(0,T;H^{s_1}_p(A))}
+\norm{w(\cdot,0)}_{B^s_{pp}(A)}\right).
\end{multline*}
\end{proposition}

\section{Formulation of the problem}\label{sec:formulation}
Let $\mathcal B$ be an elastic solid with an interior cavity completely filled by a viscous incompressible fluid $\mathcal L$. We denote by $\mathcal S$ the whole fluid-solid system. Let $\Omega$ and $\Omega_L$ be bounded domains in $\R^3$ with smooth boundaries, and with $\overline{\Omega_L}\subset\Omega$. We choose a reference configuration in which the solid occupies the domain $\Omega_B:=\Omega\setminus\overline{\Omega_L}$,  and the fluid occupies $\Omega_L$. The whole system $\mathcal S$ then occupies the bounded domain $\Omega$ in such a reference configuration. We denote $\Gamma_L:=\partial \Omega_B\cap\partial \Omega_L$, and $\Gamma_B:=\partial\Omega_B\setminus \Gamma_L$. We denote with $X$ the position vector of a material point in $\mathcal S$, having coordinates $(X_1,X_2,X_3)$ with respect a (Cartesian) coordinate system with natural basis $\{e_1,e_2,e_3\}$.

As time evolves and the solid $\mathcal B$ deforms, both fluid and solid volumes will change. We consider the {\em deformation mapping} 
\begin{equation}\label{eq:dm}
\chi(\cdot,t):\;  X\in \Omega\mapsto x=\chi(X,t)\in \R^3
\end{equation}
that gives the the position at time $t$ (also called {\em current position}) of the material point $P$ that occupied the position $X$ in the reference configuration, see Figure \ref{fig:dm}. 
\begin{figure}[h]
\includegraphics[width=\textwidth]{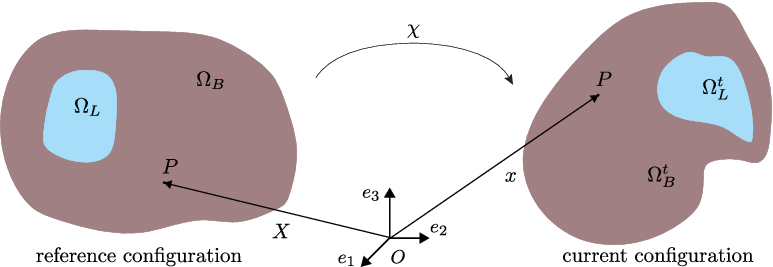}
\caption{The deformation mapping.}\label{fig:dm}
\end{figure}

The range of $\chi(\cdot,t)$ is denoted by $\Omega(t)$. More precisely, 
\[
\Omega(t):=\chi(\Omega,t)=\chi(\Omega_B,t)\cup\overline{\chi(\Omega_L,t)}\equiv\Omega_B^t\cup\overline{\Omega_L^t},\quad\text{for all }t\ge0.
\]
We refer to $\Omega_B^t$ and $\Omega_L^t$ as the {\em current} solid and fluid domain, respectively. 
We use the same frame of reference for both current and reference configuration, and we denote with $(x_1,x_2,x_3)$ the coordinates of the material point that occupied the point $(X_1,X_2,X_3)$ in the reference configuration. With respect to the basis $\{e_1,e_2,e_3\}$, we have that $x_i=\chi_i(X_1,X_2,X_3,t)$ for all $t\ge 0$ and $i=1,2,3$. 

We need the following assumptions regarding $\chi$: 
\begin{itemize}
\item[(A1)] $\chi(\cdot,0)=I$ in $\Omega_L$, where $I$ denotes the identity map. In general, $\chi(\cdot,0)\ne I$ in $\Omega_B$. 
\item[(A2)] Let $F$ be the {\em deformation gradient}: $F(X,t)=\nabla \chi(X,t)=\displaystyle\frac{\partial \chi_\alpha}{\partial X_\beta}e_\alpha\otimes e_\beta$, where $\otimes$ denotes the tensor product, defined through the dot product as $(a\otimes b)\cdot c=(b\cdot c)a$ for $a$, $b$, $c\in \R^3$. We then assume that the Jacobian 
\[
J:=\det\left[\frac{\partial \chi_\alpha}{\partial X_\beta}\right]_{(\alpha,\beta)}
>0\qquad\text{for all }(X,t)\in \overline{\Omega}\times[0,T).
\]
As a consequence of the latter, the deformation gradient $F$ is invertible, and the following equation hold 
\begin{equation}\label{eq:dt-F-1}
\frac{\partial F^{-1}}{\partial t}=-F^{-1}\cdot\nabla v\cdot F^{-1} \quad\text{in }\overline{\Omega}\times[0,T).
\end{equation}
The proof of the last displayed equation is a consequence of standard tensor calculus, see e.g. \cite[pages 12 \& 20]{ciarlet}. 
\item[(A3)] Let $u:=x-X$ denote the displacement of material points within $\mathcal S$. 
From (A1), we notice that $u(\cdot,0)\ne 0$, in general.  
\item[(A4)] The fluid is {\em incompressible}, then $J=1$ for every $(X,t)\in \overline{\Omega_L}\times [0,T)$. 
\item[(A5)] Since $J\ne 0$ in $\overline{\Omega}\times [0,T)$, the {\em cofactor tensor} $\cof{F}:=JF^{-T}$ must satisfy the following equation, known as {\em Piola condition}:  
\begin{equation}\label{eq:piola-condition}
\div (\cof{F})=\frac{\partial (\cof{F})_{\alpha\beta}}{\partial X_\beta}e_\alpha=\frac{\partial (JF^{-1}_{\beta\alpha})}{\partial X_\beta}e_\alpha=0\qquad \text{in }\overline{\Omega}\times [0,T).
\end{equation}
For a proof of the above equality we refer to \cite[end of page 39]{ciarlet}. 
\end{itemize}

The equations describing the motion of $\mathcal S$ are 
\begin{equation}\label{eq:eom}
\begin{aligned}
&\rho_B\frac{\partial^2 u}{\partial t^2}=\div P(u) &&\text{for }(X,t)\in \Omega_B\times(0,T),
\\
&\div v=0
&&\text{for }(x,t)\in \bigcup_{t\in (0,T)}\Omega_L^t\times \{t\},
\\
&\rho_{_{L}}\left(\frac{\partial v}{\partial t}+v_\alpha\frac{\partial v}{\partial x_\alpha}\right)=\div \mathcal T(v,p) &&\text{for }(x,t)\in \bigcup_{t\in (0,T)}\Omega_L^t\times \{t\},
\\
&\frac{\partial \chi}{\partial t}=v(\chi(X,t),t) &&\text{for }(X,t)\in \Omega_L\times(0,T),
\\
&P(u)\cdot \n=0 &&\text{for }(X,t)\in \Gamma_B\times(0,T),
\\
&v(\chi(X,t),t)=\frac{\partial u}{\partial t}(X,t)&&\text{for }(X,t)\in \Gamma_L\times(0,T),
\\
&P(u(X,t))\cdot \n(X,t)
\\
&\ =\mathcal T(v(\chi(X,t),t),p(\chi(X,t),t))\cdot F^{-T}(X,t)\cdot \n(X,t)
&&\text{for }(X,t)\in \Gamma_L\times(0,T). 
\end{aligned}
\end{equation}
In the above equations, $u=u(X,t)$ denotes the Lagrangian displacement of material points in $\mathcal B$, $\rho_B$ is the solid density (which we assume to be constant). 
The tensor $P(u(X,t))$ is the first Piola-Kirchhoff stress. We need a constitutive equation for such a tensor. We assume that $\mathcal B$ is  composed of a {\em linear elastic isotropic material}. The constitutive equation is given by the following {\em generalized Hooke's law}
\begin{equation}\label{eq:piola}
P(u):=\lambda(\tr(E(u))I+2\hat \mu E(u),\qquad E(u):=\frac 12 (\nabla u+(\nabla u)^T).
\end{equation}
The constants $\lambda$ and $\hat \mu$ are called Lam\'e constants, and they are both positive constants. The tensor $E(u)$ is called  (Lagrangian) infinitesimal strain tensor. Note that we will use the same symbols for differential operators --like divergence ($\div \cdot$), Laplacian ($\Delta\cdot$) and gradient ($\nabla \cdot$)-- with respect Lagrangian and Eulerian variables. It will be clear from the context with respect to which variables the differentiation is taken. For what concerns the fluid, $v=v(x,t)$ and $p=p(x,t)$ denote the Eulerian description of the fluid velocity and pressure, respectively. Due to assumption (A4), the fluid density $\rho_{_{L}}$ is constant. The tensor $\mathcal T=\mathcal T(v,p)$ denotes the (Eulerian) Cauchy stress tensor. We assume that $\mathcal L$ is a {\em viscous incompressible Newtonian fluid}. The constitutive equation for $\mathcal T$ is given by 
\begin{equation}\label{eq:cauchy}
\mathcal T(v,p):=-pI+2\mu D(v),\qquad D(v):=\frac 12 (\nabla v+(\nabla v)^T),
\end{equation}
where $\mu>0$ is the coefficient of dynamic viscosity, and $D(v)$ is called the rate of deformation tensor. Note that the the operators $E$ and $D$ in equations \eqref{eq:piola} and \eqref{eq:cauchy}, respectively, are the same differential operators but applied to different vector fields (i.e., they have different domains). Physically speaking $E(u)$ and $D(v)$ have different units, and thus different physical meanings. This explains also their different names. 

Equation \eqref{eq:eom}$_1$ are the so-called {\em Navier equations
of linear elasticity}, they represent the balance of linear momentum for $\mathcal B$. Equations \eqref{eq:eom}$_{2,3}$ are the so-called {\em Navier-Stokes equations}, and describe the conservation of mass and balance of linear momentum for $\mathcal L$. 

Concerning the boundary conditions \eqref{eq:eom}$_{5,6,7}$, they represent -in order of appearance- the traction-free condition on the outer boundary of $\mathcal B$\footnote{We assume that there is vacuum outside $\mathcal B$, and no other force is applied on the outer boundary of $\mathcal B$. }, and the continuity of velocities and tractions on the fluid-solid interface. We note that the last two conditions in \eqref{eq:eom} should be read as limits approaching $\Gamma_L$ from the exterior and the interior of the fluid, respectively. 

The driving mechanism for the deformation/motion of $\mathcal S$ will be only the initial conditions imparted on the system. The most general initial conditions are 
\begin{equation}\label{eq:ic}
\begin{aligned}
&u(\cdot,0)=u_0,\quad \frac{\partial u}{\partial t}(\cdot,0)=u_1\qquad &&\text{in }\Omega_B,
\\
&v(\cdot,0)=v_0, \qquad \chi(\cdot, 0)=I\qquad &&\text{in }\Omega_L,
\end{aligned}\end{equation}
and they must satisfy the {\em compatibility conditions}
\begin{equation}\label{eq:compatibility}
\begin{aligned}
&\div v_0=0&&\text{on }\Omega_L,
\\
&P(u_0)\cdot \n=0 &&\text{on }\Gamma_B,
\\
&v_0=u_1&&\text{on }\Gamma_L,
\\
&P(u_0)\cdot \n =\mathcal T(v_0,p_0)\cdot \n
&&\text{on }\Gamma_L. 
\end{aligned}
\end{equation}
The following {\em energy balance} characterizes ``sufficiently'' regular solutions of \eqref{eq:eom}. 
\begin{lemma}\label{lem:energy}
Smooth solutions to \eqref{eq:eom} satisfy the following {\em energy balance}:
\begin{equation}\label{eq:energy}
\frac{d}{d t}\left[\frac{\rho_{_{L}}}{2}\norm{v}^2_{\Omega_L^t}+\frac{\rho_B}{2}\norm{\frac{\partial u}{\partial t}}^2_{\Omega_B}+\frac{\lambda}{2}\norm{\div(u)}^2_{\Omega_B}+\hat \mu\norm{E(u)}^2_{\Omega_B}\right]+2\mu\norm{D(v)}^2_{\Omega_L^t}=0. 
\end{equation}
\end{lemma}

Equations \eqref{eq:eom} provide a mixed Lagrangian-Eulerian formulation of the equations of motion. In particular, the (Eulerian) velocity and pressure of the fluid vary in a domain which is time-dependent, and whose description depends on the deformation mapping $\chi$, defined in \eqref{eq:dm}, which is also an unknown of the problem (see \eqref{eq:eom}$_4$). Thus, as customary in this situation, we will solve the corresponding problem in its full Lagrangian formulation. Before proceeding, we need some notation and observations. 

For every vector field $f$ and tensor field $S$ defined on (or on a part of) $\Omega$, we write 
\[\begin{split}
f=f(x,t)=f(\chi(X,t),t)=:f^\chi(X,t),
\\
S=S(x,t)=S(\chi(X,t),t)=:S^\chi(X,t).
\end{split}\]
We may also need the so-called {\em Piola transforms} (see \cite[Exercise 1.12 \& p. 38]{ciarlet}: 
\[\begin{split}
&f_\chi(X,t):=(\cof{F(X,t)})^T\cdot f^\chi(X,t),
\\
& S_\chi(X,t):=S^\chi(X,t)\cdot\cof{F(X,t)}. 
\end{split}\] 
It follows that (see \cite[Exercise 1.12 \& Theorem 1.7-1]{ciarlet}):
\[\begin{split}
&\div_X f_\chi=J\div_x f,
\\
&\div_X S_\chi= J\div_x S. 
\end{split}\]
Using the above facts, Piola condition \eqref{eq:piola-condition}, and that $J=1$ on $\Omega_L$, we find the following Lagrangian descriptions of the divergence of the fluid velocity and stress, respectively: 
\begin{align}
&\div_x v
=\frac{\partial}{\partial X_i} [(\cof{F})_{ji}v^\chi_j]=(\cof{F})_{ji}\frac{\partial v^\chi_j}{\partial X_i}=\nabla_Xv^\chi:\cof{F}, \label{eq:piola-div-v}
\\
&\div_x \mathcal T(v,p)=\div_X \mathcal T_\chi(v^\chi,p^\chi),\label{eq:piola-div-T}
\end{align}
where 
\begin{multline}\label{eq:piola-transform-T}
\mathcal T_\chi(v^\chi,p^\chi)=-p^\chi\cof{F}+\mu\nabla_Xv^\chi\cdot(\cof{F})^T\cdot\cof{F}
\\
+\mu\cof{F}\cdot(\nabla_Xv^\chi)^T\cdot\cof{F}.
\end{multline}

The Lagrangian description of the equations of motion of $\mathcal S$ then reads as follows 
\begin{equation}\label{eq:eom-L}
\begin{aligned}
&\rho_B\frac{\partial^2 u}{\partial t^2}=\div P(u) &&\text{for }(X,t)\in \Omega_B\times(0,T),
\\
&\div v^\chi=0
&&\text{for }(X,t)\in \Omega_L\times (0,T),
\\
&\rho_{_{L}}\frac{\partial v^\chi}{\partial t}=\div \mathcal T_\chi(v^\chi,p^\chi) &&\text{for }(X,t)\in \Omega_L\times (0,T),
\\
&\frac{d \chi}{d t}=v^\chi(X,t)&&\text{for all }(X,t)\in \Omega_L\times (0,T),
\\
&P(u)\cdot \n=0 &&\text{for }(X,t)\in \Gamma_B\times(0,T),
\\
&v^\chi(X,t)=\frac{\partial u}{\partial t}(X,t)&&\text{for }(X,t)\in \Gamma_L\times(0,T),
\\
&P(u)\cdot \n=\mathcal T_\chi(v^\chi,p^\chi)\cdot \n
&&\text{for }(X,t)\in \Gamma_L\times(0,T). 
\end{aligned}
\end{equation}
It should be emphasized that, despite the above notation, the Cauchy stress tensor $\mathcal T^\chi(v^\chi,p^\chi)$) as well as the divergence free condition \eqref{eq:eom-L}$_2$ are nonlinear in $\chi$ (which remains an unknown of the problem). 

It will be more convenient to rewrite \eqref{eq:eom-L} in a more explicit form. We do so by rewriting the terms in \eqref{eq:eom-L}$_{2,3}$ exploiting the equations  \eqref{eq:piola-div-v}, \eqref{eq:piola-div-T}, \eqref{eq:piola-transform-T}, and using Piola condition \eqref{eq:piola-condition}. In addition, we drop the dependence on $\chi$ for $v$ and $p$ as it is now clear that we are working with Lagrangian variables. We also recall the hypotheses and notation introduced in (A1)--(A5), for example we recall that $F=\nabla \chi$ and $J=\det(F)=1$ on $\overline{\Omega_L}\times[0,T)$. Equations \eqref{eq:eom-L} are then equivalent to the following system of equations\footnote{All derivatives are now taken with respect to the Lagrangian variables. }
\begin{equation}\label{eq:eom-L-c}
\begin{aligned}
&\rho_B\frac{\partial^2 u}{\partial t^2}=\div P(u) &&\text{in }\Omega_B\times(0,T),
\\
&\nabla v:\cof{F}=0
&&\text{in }\Omega_L\times (0,T),
\\
&\rho_{_{L}}\frac{\partial v}{\partial t}=-\cof{F}\cdot\nabla p &&\text{in }\Omega_L\times (0,T),
\\
&\qquad\qquad\qquad\qquad+\mu\frac{\partial}{\partial X_\beta}\left[(\cof{F})_{\alpha\beta}(\cof{F})_{\alpha \gamma}\frac{\partial v}{\partial X_\gamma}\right]&&
\\
&\qquad\qquad\qquad\qquad+\mu\frac{\partial}{\partial X_j}\left[ (\cof{F})_{\ell j}(\cof{F})_{ik}\frac{\partial v_\ell}{\partial X_k}\right]e_i
\\
&\frac{d \chi}{d t}=v&&\text{in }\Omega_L\times (0,T),
\\
&P(u)\cdot \n=0 &&\text{on }\Gamma_B\times(0,T),
\\
&v=\frac{\partial u}{\partial t}&&\text{on }\Gamma_L\times(0,T),
\\
&P(u)\cdot \n=\left[-pI+\mu\nabla v \cdot(\cof{F})^T\right.
&&\text{on }\Gamma_L\times(0,T). 
\\
&\qquad\qquad\qquad\qquad\quad\quad\left. +\mu\cof{F}\cdot(\nabla v)^T\right]\cdot\cof{F}\cdot \n &&
\end{aligned}
\end{equation}
To the above equations, we append the following initial conditions 
\begin{equation}\label{eq:ic-L}
\begin{aligned}
&u(\cdot,0)=u_0,\quad \frac{\partial u}{\partial t}(\cdot,0)=u_1,\qquad &&\text{in }\Omega_B,
\\
&v(\cdot,0)=v_0, \qquad \chi(\cdot,0)=I,\qquad &&\text{in }\Omega_L,
\end{aligned}\end{equation}
satisfying ``{\em suitable}'' {\em compatibility conditions} (see conditions {\em (i)--(iv)} in Theorem \ref{th:main}).

\section{Main theorem and strategy of its proof}\label{sec:main-th}
Objective of this paper is to prove the local well-posedness of \eqref{eq:eom-L-c}. The following is the main theorem of this paper. 
\begin{theorem}\label{th:main}
Let $(u_0,u_1,v_0)\in H^{3}(\Omega_B)\times H^{3/2}(\Omega_B) \times H^{5/2}(\Omega_L)$ satisfying the {\em compatibility conditions}:  
\begin{itemize}
\item[(i)] $P(u_0)\cdot\n=0$ on $\Gamma_B$, and $u_1=v_0$ on $\Gamma_L$;
\item[(ii)] $\div v_0=0$ on $\Omega_L$;
\item[(iii)] There exists $u_2\in H^1(\Omega_B)$ such that $\rho_Bu_2=\div P(u_0)$ in $\Omega_B$;
\item[(iv)] There exist $(v_1,q_0)\in H^{1}(\Omega_L)\times H^{3/2}(\Omega_L)$ such that 
\begin{equation}\label{eq:compatibility-main}
\begin{aligned}
&\div v_1=\nabla v_0 :(\nabla v_0)^T&&\text{in }\Omega_L,
\\
&\rho_Lv_1=-\nabla q_0+\mu \Delta v_0&&\text{in }\Omega_L,
\\
&v_1=u_2 &&\text{on }\Gamma_L,
\\
&\mathcal T(v_0,q_0)\cdot \n=P(u_0)\cdot\n&&\text{on }\Gamma_L.
\end{aligned}
\end{equation}
\end{itemize}
Then, there exists $T=T(u_0,u_1,v_0)>0$ such that equations \eqref{eq:eom-L-c}--\eqref{eq:ic-L} admit a unique strong solution $(u,v,p,\chi)$ satisfying 
\[\begin{aligned}
&u\in \{\xi\in C([0,T];H^{5/2}(\Omega_B))\cap C^1([0,T];H^{3/2}(\Omega_B))\cap C^2([0,T];H^{1/2}(\Omega_B)):
\\
&\qquad\qquad\qquad\qquad\qquad\qquad\qquad\qquad\qquad P(u)\cdot \n\in 
H^{3/2,3/2}(\partial \Omega_B\times(0,T))\},
\\
&v\in L^2(0,T;H^{3}(\Omega_L))\cap H^1(0,T;H^2(\Omega_L))\cap H^2(0,T;L^2(\Omega_L)),
\\
&p\in \{\pi\in L^2(0,T;H^{2}(\Omega_L))\cap H^1(0,T;H^1(\Omega_L)):\; \frac{\partial \pi}{\partial t}\in H^{1/4}(0,T;L^2(\Gamma_L))\},
\\
&\chi\in H^1(0,T;H^3(\Omega_L))\cap H^2(0,T;H^2(\Omega_L))
\cap H^3(0,T;L^2(\Omega_L)).
\end{aligned}\]
In addition, $\chi(\cdot,t):\; \Omega_L\to\Omega_L^t$ is a volume preserving diffeomorphism for all $t\in (0,T)$, and satisfies  
$\div ((\nabla \chi)^{-T})=0$ in $\Omega_L\times(0,T)$. 
\end{theorem}

\begin{remark}\label{rm:compatibility-conditions}
The compatibility conditions $(i)$--$(iv)$ are required because of the regularity of the data and of the corresponding solution. In particular, by Proposition \ref{prop:traceat0}, our solution satisfies 
\[\begin{split}
&\frac{\partial \chi}{\partial t}=v\in C([0,T];H^{5/2}(\Omega_L)),\qquad 
\\
&\frac{\partial v}{\partial t}\in C([0,T];H^1(\Omega_L)),\quad p\in C([0,T];H^{3/2}(\Omega_L)). 
\end{split}\]
So, equations \eqref{eq:eom-L-c} must be satisfied at time $t=0$, this immediately explains the compatibility conditions $(i)$--$(iii)$ and the equations \eqref{eq:compatibility-main}$_{2,3,4}$. In particular, 
\[
u_2=\frac{\partial^2 u}{\partial t^2}(\cdot,0),\qquad v_1=\frac{\partial v}{\partial t}(\cdot,0),\qquad q_0=p(\cdot, 0).
\]
We further notice that $\nabla v:\cof{F}\in C^1([0,T];L^2(\Omega_L))$,  
thus
\[
\frac{\partial }{\partial t}[\nabla v:\cof{F}]|_{t=0}=0.
\]
From the above equality and Piola condition \eqref{eq:piola-condition}, we find \eqref{eq:compatibility-main}$_1$.  

We observe that $\nabla v_0:(\nabla v_0)^T\in L^2(\Omega_L)$ with $\norm{\nabla v_0:(\nabla v_0)^T}_{L^2(\Omega_L)}\le c\norm{v_0}^2_{H^2(\Omega_L)}$ (thanks to H\"older's inequality and Sobolev embedding theorem). Let $U_2\in H^1(\Omega_L)$ be an extension of $u_2|_{\Gamma_L}$ in $\Omega_L$ given by \cite[Chapter 1, Theorem 8.3]{lionsmagenesI}, and consider $w:=v_1-U_2$ the solution to the following problem  
\begin{equation}\label{eq:div-w}\begin{aligned}
&\div w=\nabla v_0:(\nabla v_0)^T+\div U_2 \qquad &&\text{in }\Omega_L,
\\
&w=0 \qquad &&\text{on }\Gamma_L.
\end{aligned}\end{equation}
Note that, by assuming \eqref{eq:compatibility-main}$_{1,3}$, we necessarily have that 
\[
\int_{\Gamma_L}\n\cdot\left[(\nabla v_0)\cdot v_0\right]=\int_{\Gamma_L}u_2\cdot \n.
\]
By \cite[Theorem III.3.1]{Ga}, Poincar\'e inequality, and continuity of the right-inverse of the trace (see again \cite[Chapter 1, Theorem 8.3]{lionsmagenesI}), we infer that there exists at least one solution $w\in H^1_0(\Omega_L)$ of \eqref{eq:div-w} satisfying the estimate 
\[
\norm{w}_{H^1(\Omega_L)}\le c\left(\norm{v_0}^2_{H^2(\Omega_L)}+\norm{u_2}_{H^{1/2}(\Gamma_L)}\right),
\]
with a positive constant $c$ depending only on $\Omega_L$. Recalling the definition of $w$ and compatibility condition $(iii)$, by possibly updating the constant $c$, we obtain that a solution $v_1\in H^1(\Omega_L)$ to \eqref{eq:compatibility-main}$_{1,3}$ must satisfy the following estimate  
\begin{equation}\label{eq:estimate-compatibility-v1}
\norm{v_1}_{H^1(\Omega_L)}\le c\left(\norm{v_0}^2_{H^2(\Omega_L)}+\norm{u_0}_{H^{3}(\Omega_B)}\right).
\end{equation}
Thanks to the above estimate, we can also provide an estimate (in terms of the initial data) for the corresponding pressure field $q_0$ satisfying \eqref{eq:compatibility-main}$_{2,4}$. In fact, by the Helmholtz-Weyl decomposition of $L^2(\Omega_L)$ (see \cite[Theorem III.1.1]{Ga}), we can write 
\[
\mu\Delta v_0=w_1+w_2,
\]
where $w_1=\nabla \pi$ for some $\pi\in H^1(\Omega_L)$ and $w_2\in L^2_\sigma(\Omega_L)$, where $L^2_\sigma(\Omega_L)=\{\varphi\in L^2(\Omega_L):\; \div\varphi=0\text{ in }\Omega_L\text{ and }\varphi\cdot\n=0\text{ on }\Gamma_L\}$. Set $p_0:=q_0-\pi$, then $p_0$ is a weak (in the sense of distributions) solution to the following elliptic problem 
\[\begin{aligned}
&\Delta p_0=-\rho_L\nabla v_0:(\nabla v_0)^T\qquad&&\text{in }\Omega_L,
\\
&p_0=\left[2\mu D(v_0)-P(u_0)\right]\cdot\n-\pi\qquad&&\text{on }\Gamma_L.
\end{aligned}\] 
By elliptic estimates, we have that 
\[
\norm{p_0}_{H^1(\Omega_L)}\le c\left(\norm{v_0}^2_{H^{2}(\Omega_L)}+\norm{v_0}_{H^{2}(\Omega_L)}+\norm{u_0}_{H^3(\Omega_B)}\right),
\]
and since 
\[
\nabla q_0=\rho_L v_1-\mu\Delta v_0\in H^{1/2}(\Omega_L),
\]
we infer that 
\begin{equation}\label{eq:estimate-compatibility-q0}
\norm{q_0}_{H^{3/2}(\Omega_L)}\le c\left(\norm{v_0}^2_{H^{2}(\Omega_L)}+\norm{v_0}_{H^{5/2}(\Omega_L)}+\norm{u_0}_{H^3(\Omega_B)}\right). 
\end{equation}
\end{remark}

To prove the existence of solutions to the above initial-boundary value problem, we will use the following fixed point argument. 
We denote 
\begin{equation}\label{eq:spaceD}
\begin{split}
\mathcal D&:=H^1(0,T;H^3(\Omega_L))
\cap H^2(0,T;H^2(\Omega_L))\cap H^3(0,T;L^2(\Omega_L)),
\\
\mathcal D_1&:=\mathcal D \cap W^{1,\infty}(0,T;H^{5/2}(\Omega_L))\cap W^{2,\infty}(0,T;H^1(\Omega_L)),
\end{split}
\end{equation}
endowed with the norms 
\begin{equation}\label{eq:normD}
\begin{split}
&\norm{\xi}_{\mathcal D}:=\norm{\xi}_{H^1(0,T;H^3(\Omega_L))} + \norm{\xi}_{H^2(0,T;H^2(\Omega_L))} + \norm{\xi}_{H^3(0,T;L^2(\Omega_L))}
\\
&\norm{\xi}_{\mathcal D_1}:=\norm{\xi}_{\mathcal D}+\norm{\xi}_{W^{1,\infty}(0,T;H^{5/2}(\Omega_L))}+\norm{\xi}_{W^{2,\infty}(0,T;H^1(\Omega_L))}. 
\end{split}\end{equation}
Let $R>0$ be sufficiently large so that $\norm{v_0}_{H^{5/2}(\Omega_L)}\le R$, and the following set is non-empty  
\begin{multline}\label{eq:B_R}
B_R=\{\xi\in \mathcal D_1:\; \xi(\cdot,0)=I,\ \frac{\partial \xi}{\partial t}(\cdot,0)=v_0,\ 
\det(\nabla \xi)\ne 0\text{ in }\Omega_L\times(0,T),
\\
\norm{\xi}_{\mathcal D_1}\le R\}. 
\end{multline}

Fix $\hat \chi\in B_R$, and denote by $\hat F=\nabla \hat \chi$ the corresponding deformation gradient (recall that $\hat J=\det(\hat F)\ne 0$ on $\Omega_L$, and then \eqref{eq:piola-condition} is satisfied). Furthermore, suppose that 
\[
\hat \chi (\cdot,0)= I\qquad\text{and}\qquad \frac{\partial \hat \chi}{\partial t}(\cdot,0)=v_0.
\]
We seek a solution $(u,v,p)$ of 
\begin{equation}\label{eq:eom-L-c-hat}
\begin{aligned}
&\rho_B\frac{\partial^2 u}{\partial t^2}=\div P(u) &&\text{in }\Omega_B\times(0,T),
\\
&\nabla v:\cof{\hat F}=0 
&&\text{in }\Omega_L\times (0,T),
\\
&\rho_{_{L}}\frac{\partial v}{\partial t}=-\cof{\hat F}\cdot\nabla p && \text{in }\Omega_L\times (0,T),
\\
&\qquad\qquad\qquad\qquad+\mu\frac{\partial}{\partial X_\beta}\left[(\cof{\hat F})_{\alpha\beta}(\cof{\hat F})_{\alpha \gamma}\frac{\partial v}{\partial X_\gamma}\right]&&
\\
&\qquad\qquad\qquad\qquad+\mu\frac{\partial}{\partial X_j}\left[ (\cof{\hat F})_{\ell j}(\cof{\hat F})_{ik}\frac{\partial v_\ell}{\partial X_k}\right]e_i
\\
&P(u)\cdot \n=0 &&\text{on }\Gamma_B\times(0,T),
\\
&v=\frac{\partial u}{\partial t}&&\text{on }\Gamma_L\times(0,T),
\\
&P(u)\cdot \n=\left[-pI+\mu\nabla v \cdot(\cof{\hat F})^T \right. && \text{on }\Gamma_L\times(0,T),
\\
&\qquad\qquad\qquad\qquad\qquad \left. +\mu\cof{\hat F}\cdot(\nabla v)^T\right]\cdot\cof{\hat F}\cdot \n,
&&
\end{aligned}
\end{equation}
with $P(u)$  the first Piola-Kirchhoff stress tensor introduced in \eqref{eq:piola}, and
with initial conditions 
\begin{equation}\label{eq:ic-hat}
\begin{aligned}
&u(\cdot, 0)=u_0,\qquad \frac{\partial u}{\partial t}(\cdot, 0)=u_1 &&\text{in }\Omega_B,
\\
&v(\cdot, 0)=V_0 &&\text{in }\Omega_L.
\end{aligned}
\end{equation}
These initial conditions need to satisfy suitable compatibility conditions (see {\em (i)--(iv)} in Theorem \ref{th:existence-eom-L-c-hat}). 
Once $(u,v,p)$ is found, consider the map  
\begin{equation}\label{eq:fixed-point-map-chi}
\Lambda:\; \hat \chi \mapsto \Lambda(\hat \chi)=\chi(X,t):=X+\int^t_0v(X,s)\; ds.
\end{equation}
A fixed point of the latter will provide the solution $(u,v,p,\chi)$ to \eqref{eq:eom-L-c}. 

Before we embark on the proof of existence of solutions to \eqref{eq:eom-L-c-hat}, and then of \eqref{eq:eom-L-c},  
let us set up some notation and prove some useful estimates concerning the deformation mapping $\hat \chi$. 

With respect to the natural basis associated with the Cartesian coordinates $(X_1,X_2,X_3)$, we can represent the deformation gradient as $\hat F=\hat F_{\alpha\beta} e_\alpha\otimes e_\beta$, for some functions $\hat F_{\alpha\beta}=\hat F_{\alpha\beta}(X_1,X_2,X_3,t)$, $\alpha,\beta=1,2,3$, such that 
\[
\hat F_{\alpha\beta}(X_1,X_2,X_3,t)=\delta_{\alpha\beta}+\int^t_0\frac{\partial }{\partial s}\frac{\partial \hat \chi_\alpha}{\partial X_\beta}(X_1,X_2,X_3,s)\; d s.
\] 
The third-order tensor defining the gradient of $\hat F$ and the fourth-order tensor of second derivatives of ${\hat F}$ are denoted as follows 
\[\begin{split}
\nabla \hat F&=\frac{\partial {\hat F}_{\alpha\beta}}{\partial X_\gamma} e_\alpha\otimes e_\beta\otimes  e_\gamma,\qquad
D^2 {\hat F}=\frac{\partial^2 {\hat F}_{\alpha\beta}}{\partial X_\delta\partial X_\gamma} e_\alpha\otimes e_\beta\otimes  e_\gamma\otimes e_\delta.
\end{split}\]
 By Jensen's inequality, we have that  
\begin{equation}\label{eq:iv0.1}\begin{split}
\norm{\nabla {\hat F}}_{L^2(\Omega_L)}
&\le T^{1/2}\norm{\hat\chi}_{H^1(0,T;H^2(\Omega_L))},
\end{split}\end{equation}
and similarly
\begin{equation}\label{eq:iv0.2}\begin{split}
\norm{D^2 {\hat F}}_{L^2(\Omega_L)}
&\le T^{1/2}\norm{\hat\chi}_{H^1(0,T;H^3(\Omega_L))}.
\end{split}\end{equation}

Furthermore, if we set $\hat v=\partial \hat \chi/\partial t$, we have that 
\[
D^2\hat v\in L^2(0,T;H^1(\Omega_L))\cap L^\infty(0,T;L^2(\Omega_L)),
\] 
and by Lemma \ref{lem:holder-st} (with $p=2$, $s_1=1$, $q=+\infty$, $s_2=0$, and $\theta=2/3$), we find that 
\[
\int^t_0\nabla(\nabla \hat v)\;d s=D^2\hat \chi\in C([0,T];L^3(\Omega_L))\text{ with }\norm{D^2\hat \chi}_{C([0,T];L^3(\Omega_L))}\le CT^{2/3}R,
\]
where $C>0$ is a constant independent of $T$. 

Thanks to the above observations and the following facts \footnote{Throughout  the paper, $\nabla {\hat F}^{-1}:=\nabla G({\hat F})$ and $D^2{\hat F}^{-1}=D^2 G({\hat F})$, where $G$ is the tensor field defined as $G({\hat F})={\hat F}^{-1}$. Moreover, the $\alpha\beta$-component of ${\hat F}^{-1}$ with respect to the basis $\{e_\alpha\otimes e_\beta:\; \alpha,\beta=1,2,3\}$ will be denoted with ${\hat F}^{-1}_{\alpha\beta}$. We also recall that $\cof{\hat F}=\hat J\hat F^{-T}$, $\hat J=\det(\hat F)$, $\frac{d \hat J}{d t}=\nabla \hat v:\cof{\hat F}$ (by \cite[Section 10]{gurtin}  and \eqref{eq:piola-div-v}), and $\hat F$ satisfies \eqref{eq:dt-F-1}.  }
\[\begin{split}
&\frac{\partial {\hat F^{-1}}_{\alpha\beta}}{\partial X_\gamma}=-{\hat F}^{-1}_{\alpha\delta}\frac{\partial {\hat F}_{\delta\epsilon}}{\partial X_\gamma}{\hat F}^{-1}_{\epsilon\beta},\qquad \text{for all $\alpha,\beta,\gamma=1,2,3$, }
\\
&\frac{\partial }{\partial t}\cof{\hat F}=\left[(\nabla \hat v:\cof{\hat F})I-\cof{\hat F}\cdot(\nabla \hat v)^T\right]\cdot \hat F^{-T},
\end{split}\]
by repeatedly using Sobolev embeddings, H\"older's inequality, and the trace theorem (see e.g. \cite[Subsection 2.1]{boulakia19} for similar estimates), we can easily obtain the proof of the following lemma. 

\begin{lemma}\label{lem:estimate_hatchi}
Let $\hat\chi\in B_R$, and denote by ${\hat F}=\nabla \hat\chi$ the corresponding deformation gradient. Then, there exists a positive constant $C$, independent of $T$, such that following estimates hold:
\begin{itemize}
\item[(i)] $\norm{\hat F-I}_{C([0,T];H^2(\Omega_L))}$, $\norm{{\hat F}-I}_{C([0,T]\times\overline{\Omega_L})}\le CT^{1/2}R$.
\item[(ii)] $\displaystyle\norm{\frac{\partial \chi}{\partial t}}_{L^{2/\theta}(0,T;H^{2+\theta}(\Omega_L))}\le CR$ for all $\theta\in [0,1]$. 
\item[(iii)] $\norm{\nabla{\hat F}}_{C([0,T];L^3(\Omega_L))}\le CT^{2/3}R$.
\item[(iv)] If $T^{1/2}R\le 1$, then 
\begin{equation}\label{eq:iv1}
\begin{split}
&\norm{{\hat F}}_{C([0,T]\times\overline{\Omega_L})}, \qquad  \norm{\cof{\hat F}}_{C([0,T]\times\overline{\Omega_L})},
\\ 
&\norm{{\hat F}}_{C([0,T];H^2(\Omega_L))}, \qquad \norm{\cof{\hat F}}_{C([0,T];H^2(\Omega_L))}\le C;
\end{split}
\end{equation}
\begin{equation}\label{eq:iv2}
\norm{\cof{\hat F}-I}_{C([0,T]\times\overline{\Omega_L})}\le CT^{1/2}R;
\end{equation}
\begin{equation}\label{eq:iv3}
\norm{\nabla \cof{\hat F}}_{C([0,T];L^3(\Omega_L))}\le CT^{2/3}R.
\end{equation}
\begin{equation}\label{eq:iv4}
\norm{\cof{\hat F}-I}_{H^1(0,T;C(\Gamma_L))},\quad \norm{\cof{\hat F}\cdot(\cof{\hat F}-I)}_{H^1(0,T;C(\Gamma_L))} \le CR. 
\end{equation}
\begin{equation}\label{eq:iv4.1}
\norm{\nabla \cof{\hat F}}_{C([0,T];L^2(\Omega_L))},\quad \norm{D^2 \cof{\hat F}}_{C([0,T],L^2(\Omega_L))}\le CT^{1/2}R.
\end{equation}
\begin{equation}\label{eq:iv5}
\norm{\frac{\partial }{\partial t}\nabla \cof{\hat F}}_{L^2(0,T;L^3(\Omega_L))}\le CRT^{1/6}.
\end{equation}
\begin{equation}\label{eq:iv6}
\norm{\frac{\partial }{\partial t}\cof{\hat F}}_{L^\infty(0,T;H^{3/2}(\Omega_L))},\quad \norm{\frac{\partial^2 }{\partial t^2}\cof{\hat F}}_{L^\infty(0,T;L^2(\Omega_L))}\le CR. 
\end{equation}
\begin{multline}\label{eq:iv7}
\norm{\frac{\partial}{\partial t}\cof{\hat F}}_{H^{1/4}(0,T;L^4(\Gamma_L))}, 
\norm{\cof{\hat F}\cdot\frac{\partial}{\partial t}\cof{\hat F}}_{H^{1/4}(0,T;L^4(\Gamma_L))}\le CRT^{1/12}. 
\end{multline}
\end{itemize}
\end{lemma}

\section{Existence of solutions to \eqref{eq:eom-L-c-hat}-\eqref{eq:ic-hat}}\label{sec:proof-thm-eom-L-c-hat}\label{sec:proof-main0}
In this section, $\hat \chi\in B_R$ is fixed. Correspondingly, we have $\hat F=\nabla \hat \chi$, and suppose that 
\[
\frac{\partial \hat \chi}{\partial t}(\cdot,0)=v_0,
\]
with $\norm{v_0}_{H^{5/2}(\Omega_L)}\le R$. 

We use the following notation 
for the functional spaces  
\begin{equation*}
\begin{split}
&\mathcal V_2:=L^2(0,T;H^3(\Omega_L))\cap H^1(0,T;H^2(\Omega_L)) \cap H^2(0,T;L^2(\Omega_L))
\\
&\Pi_2:=L^2(0,T;H^2(\Omega_L))\cap H^1(0,T;H^1(\Omega_L))\cap H^{1+1/4}(0,T;L^2(\Gamma_L)). 
\end{split}\end{equation*}
with norms defined as follows 
\begin{equation}\label{eq:normV2}
\norm{v}_{{\mathcal V}_2}:=\norm{v}_{L^2(0,T;H^3(\Omega_L))}+\norm{\frac{\partial v}{\partial t}}_{L^2(0,T;H^2(\Omega_L))}+\norm{\frac{\partial^2 v}{\partial t^2}}_{L^2(0,T;L^2(\Omega_L))}
\end{equation}
and 
\begin{equation}\label{eq:normPi2}
\norm{p}_{\Pi_2}:=\norm{p}_{L^2(0,T;H^2(\Omega_L))}+\norm{\frac{\partial p}{\partial t}}_{L^2(0,T;H^1(\Omega_L))}+\norm{\frac{\partial p}{\partial t}}_{H^{1/4}(0,T;L^2(\Gamma_L))}. 
\end{equation}
Accordingly, we denote by  
\[\begin{split}
\mathcal U:=\{\xi\in C(0,T;H^{5/2}(\Omega_B))\cap & C^1(0,T;H^{3/2}(\Omega_B))\cap C^2(0,T;H^{1/2}(\Omega_B)):
\\
&\qquad\quad P(u)\cdot \n\in  
H^{3/2,3/2}(\partial \Omega_B\times(0,T))\},
\end{split}\]
endowed with the norm 
\begin{multline}\label{eq:normU}
\norm{u}_{\mathcal U}:=\norm{u}_{C([0,T];H^{5/2}(\Omega_B))}+\norm{\frac{\partial u}{\partial t}}_{C([0,T]; H^{3/2}(\Omega_B))}+\norm{\frac{\partial^2 u}{\partial t^2}}_{C([0,T]; H^{1/2}(\Omega_B))}.
\\
+\norm{P(u)\cdot \n}_{H^{3/2,3/2}(\Gamma_L\times(0,T))}.
\end{multline}

We are now ready to state the main result of this section. 
\begin{theorem}\label{th:existence-eom-L-c-hat}
Let $(u_0,u_1,V_0)\in H^{3}(\Omega_B)\times H^{3/2}(\Omega_B) \times H^{5/2}(\Omega_L)$ satisfying the following {\em compatibility conditions}: 
\begin{itemize}
\item[(i)] $P(u_0)\cdot\n=0$ on $\Gamma_B$, and $V_0=u_1$ on $\Gamma_L$;
\item[(ii)] $\div V_0=0$ on $\Omega_L$;
\item[(iii)]  There exists $u_2\in H^1(\Omega_B)$ such that $\rho_Bu_2=\div P(u_0)$ in $\Omega_B$;
\item[(iv)] There exist $(V_1,Q_0)\in H^1(\Omega_L)\times H^{3/2}(\Omega_L)$ 
such that 
\begin{equation}\label{eq:compatibility-main1}
\begin{aligned}
&\div V_1=\nabla V_0 :(\nabla v_0)^T&&\text{in }\Omega_L,
\\
& V_1=u_2&&\text{on }\Gamma_L,
\\
&\rho_L V_1=-\nabla Q_0+\mu \Delta V_0&&\text{in }\Omega_L,
\\
&\mathcal T(V_0,Q_0)\cdot n=P(u_0)\cdot\n&&\text{on }\Gamma_L.
\end{aligned}
\end{equation}
\end{itemize}
Then, there exists a non-decreasing positive function $G=G(t)$, $G(0)=0$, $G(t)\to +\infty$ as $t\to +\infty$ such that for all $T\in (0,1)$ satisfying  
\begin{equation}\label{eq:small-time-R}
G(T)(R+1)<1,
\end{equation}
the problem \eqref{eq:eom-L-c-hat}-\eqref{eq:ic-hat} admits a unique solution $(u,v,p)\in \mathcal U\times\mathcal V_2\times \Pi_2$ 
with 
\begin{multline}\label{eq:energy_hateq}
\norm{u}_{\mathcal U}+\norm{v}_{\mathcal V_2}+\norm{p}_{\Pi_2}\le C\left[(1+\norm{v_0}_{H^{5/2}(\Omega_L)})\norm{V_0}_{H^{5/2}(\Omega_L)}\right.
\\
\left.+\norm{u_0}_{H^{3}(\Omega_B)}+\norm{u_1}_{H^{3/2}(\Omega_B)}\right],
\end{multline}
where $C$ is a positive constant independent of $T$. 

In addition, there exists $R\ge \norm{V_0}_{H^{5/2}(\Omega_L)}$ such that 
\[
\chi(\cdot,t):=X+\int^t_0v(\cdot,s)\; ds \in B_R.
\]
\end{theorem}

\begin{remark}\label{rm:compatibility-conditions-hatchi}
The compatibility conditions $(i)$--$(iv)$ are required because of the regularity of the data and of the corresponding solution. Following the same arguments as in Remark \ref{rm:compatibility-conditions}, we find that 
\begin{equation}\label{eq:estimate-compatibility-V1}
\norm{V_1}_{H^1(\Omega_L)}\le c\left(\norm{V_0}_{H^2(\Omega_L)}\norm{v_0}_{H^2(\Omega_L)}+\norm{u_0}_{H^{3}(\Omega_B)}\right),
\end{equation}
and 
\begin{equation}\label{eq:estimate-compatibility-Q0}
\norm{Q_0}_{H^{3/2}(\Omega_L)}\le c\left[\norm{V_0}_{H^{5/2}(\Omega_L)}\left(\norm{v_0}_{H^2(\Omega_L)}+1\right)+\norm{u_0}_{H^3(\Omega_B)}\right]. 
\end{equation}
\end{remark}

To prove Theorem \ref{th:existence-eom-L-c-hat}, we will use a fixed point argument. Let $(u_0,u_1,V_0)\in H^{3}(\Omega_B)\times H^{3/2}(\Omega_B) \times H^{5/2}(\Omega_L)$ satisfying the compatibility conditions $(i)$--$(iv)$, with $(V_1,Q_0)\in H^1(\Omega_L)\times H^{3/2}(\Omega_L)$, in Theorem \ref{th:existence-eom-L-c-hat}. 
We consider 
\begin{equation}\label{eq:B_M}
\mathcal E:=\{(w,\pi)\in \mathcal V_2\times \Pi_2:\; w(\cdot,0)=V_0,\ \frac{\partial w}{\partial t}(\cdot,0)=V_1,\ \pi(\cdot,0)=Q_0,\}.
\end{equation}

In the following lemma, we summarize some additional regularity properties for $(\tilde v,\tilde p)\in \mathcal V_2\times \Pi_2$. 

\begin{lemma}\label{lem:r-tilde-vq}
Let $(\tilde v,\tilde p)\in \mathcal E$ and $T\in (0,1)$. Then, 
\[\begin{split}
&\tilde v\in C([0,T];H^{5/2}(\Omega_L)),\quad \tilde p\in C([0,T];H^{3/2}(\Omega_L)),\quad 
\frac{\partial }{\partial t}\nabla\tilde v\in C([0,T];L^2(\Omega_L)),
\\
&\nabla\tilde v\in H^1(0,T;L^4(\Gamma_L)),\qquad \tilde v\in H^{7/4}(\Gamma_L\times (0,T)).
\end{split}\]
In addition, the following estimates hold with a constant $C>0$ independent of $T$:
\begin{align}
&\norm{\tilde v}_{C([0,T];H^2(\Omega_L))}\le C\left[T^{1/2}\norm{\tilde v}_{\mathcal V_2}
+\norm{V_0}_{H^2(\Omega_L)}\right]\label{eq:tildev0}
\\
&\norm{\tilde v}_{C([0,T];H^{5/2}(\Omega_L))}\le C\left[\norm{\tilde v}_{\mathcal V_2}
+\norm{V_0}_{H^{5/2}(\Omega_L)}\right]\label{eq:tildev1}
\\
&\norm{\tilde p}_{C([0,T];H^{3/2}(\Omega_L))}\le C\left[\norm{\tilde p}_{\Pi_2}
+\norm{Q_0}_{H^{3/2}(\Omega_L)}\right];\label{eq:tildep1}
\\
&\norm{\frac{\partial}{\partial t}\nabla \tilde v}_{C([0,T];L^2(\Omega_L))}\le C\left[\norm{\tilde v}_{\mathcal V_2}
+\norm{V_1}_{H^1(\Omega_L)}\right];\label{eq:tildev2}
\\
&\norm{\nabla \tilde v}_{H^1(0,T;L^4(\Gamma_L))}\le C\norm{\tilde v}_{H^1(0,T;H^2(\Omega_L))};\label{eq:tildev3}
\\
&\norm{\frac{\partial \tilde v}{\partial t}}_{H^{1/4}(0,T;H^1(\Gamma_L))}\le C \norm{\tilde v}_{\mathcal V_2}, \label{eq:tildev4.0}
\\
&\norm{\tilde v}_{H^{1/2}(0,T;H^{7/4}(\Gamma_L))}\le C\left(T^{1/4}\norm{\tilde v}_{\mathcal V_2}+T^{1/8}\norm{V_0}_{H^{5/2}(\Omega_L)}\right), \label{eq:tildev5}
\\
&\norm{\frac{\partial \tilde v}{\partial t}}_{H^{1/2}(0,T;L^2(\Gamma_L))}\le C\left(T^{1/8}\norm{\tilde v}_{\mathcal V_2}+T^{1/8}\norm{V_1}_{H^1(\Omega_L)}\right). \label{eq:tildev6}\end{align}
\end{lemma} 

\begin{proof}
Under the stated regularity of $\tilde v$ and $\tilde p$, properties \eqref{eq:tildev0}-\eqref{eq:tildep1} immediately follow from Proposition \ref{prop:traceat0}. In addition, 
\[
\frac{\partial \tilde v}{\partial t}\in L^2(0,T;H^2(\Omega_L))\cap H^1(0,T;L^2(\Omega_L)),
\]
and again by Proposition \ref{prop:traceat0}, we have that $\partial \tilde v/\partial t\in C([0,T];H^1(\Omega_L))$ and 
\[
\norm{\frac{\partial }{\partial t}\nabla \tilde v}_{C([0,T];L^2(\Omega_L))}\le C_1\left[\norm{\tilde v}_{\mathcal V_2}+\norm{\frac{\partial}{\partial t}\nabla \tilde v(\cdot,0)}_{L^2(\Omega_L)}\right].
\]
Since by hypothesis $\nabla \tilde v\in H^1(0,T;H^1(\Omega_L))$, by the trace theorem and Sobolev embedding, we find that 
\[
\nabla \tilde v\in H^1(0,T;L^4(\Gamma_L))
\quad\text{with }\norm{\nabla \tilde v}_{H^1(0,T;L^4(\Gamma_L))}\le C_2\norm{\tilde v}_{\mathcal V_2}.
\]
By Remark \ref{rm:trace-regularity-v} (with $\theta=1/4$) and the trace theorem, we immediately have the estimate \eqref{eq:tildev4.0} for $\partial \tilde v/\partial t$ on $\Gamma_L$. 

Finally, by Remark \ref{rm:trace-regularity-v} and since $\tilde v\in \mathcal V_2$, we have that 
\[\begin{aligned}
&\tilde v-V_0\in {}_0H^{3/4}(0,T;H^{7/4}(\Gamma_L)) 
\\
&\frac{\partial \tilde v}{\partial t}-V_1\in {}_0H^{5/8}(0,T;H^{1/4}(\Gamma_L)) 
\end{aligned}\]
where we recall that ${}_0H^\sigma(0,T;\mathcal X):=\{f=f(x,t)\in H^\sigma(0,T;\mathcal X):\; f(\cdot,0)=0\}$ for $\sigma\in (1/2,3/2)$. Apply Lemma \ref{lem:tsigma-s} with $\sigma=3/4$ and $s=1/2$ to $\tilde v-V_0$, and with $\sigma=5/8$ and $s=1/2$ to $\displaystyle\frac{\partial \tilde v}{\partial t}-V_1$, respectively, we find 
\[\begin{split}
\norm{\tilde v-V_0}_{H^{1/2}(0,T;H^{7/4}(\Gamma_L))}&\le C_3T^{1/4}\norm{\tilde v-V_0}_{H^{3/4}(0,T;H^{7/4}(\Gamma_L))},
\\
\norm{\frac{\partial \tilde v}{\partial t}-V_1}_{H^{1/2}(0,T;L^2(\Gamma_L))}&\le C_4 T^{1/8}\norm{\frac{\partial \tilde v}{\partial t}-V_1}_{H^{5/8}(0,T;L^2(\Gamma_L))},
\end{split}\]
for some positive constants $C_3$ and $C_4$ independent of $T$. Since 
\[
\norm{V_0}_{H^{1/2}(0,T;H^{7/4}(\Gamma_L))}=T^{1/2}\norm{V_0}_{H^{7/4}(\Gamma_L)}\le C_5T^{1/2}\norm{V_0}_{H^{5/2}(\Omega_L)}
\]
and 
\[
\norm{V_1}_{H^{5/8}(0,T;L^2(\Gamma_L))}=T^{1/2}\norm{V_1}_{L^2(\Gamma_L)}\le C_6T^{1/2}\norm{V_1}_{H^{1}(\Omega_L)},
\]
we conclude that 
\[\begin{split}
\norm{\tilde v}_{H^{1/2}(0,T;H^{7/4}(\Gamma_L))}&\le C_7\left(T^{1/4}\norm{\tilde v}_{\mathcal V_2}+T^{1/8}\norm{V_0}_{H^{5/2}(\Omega_L)}\right),
\\
\norm{\frac{\partial \tilde v}{\partial t}}_{H^{1/2}(0,T;L^2(\Gamma_L))}&\le C_8\left(T^{1/8}\norm{\tilde v}_{\mathcal V_2}+T^{1/8}\norm{V_1}_{H^1(\Omega_L)}\right),
\end{split}\]
where $C_i$, for $i=5,...,8$ are positive constants independent of $T$, and we have used the fact that $T\in (0,1)$. This proves the estimates \eqref{eq:tildev5} and \eqref{eq:tildev6}. Finally, by Lemma \ref{lem:interpolation} part $(I1)$, we have that 
\[
\tilde v\in H^\theta(0,T;H^m(\Gamma_L))\qquad\text{for all }\theta\in [0,1]\text{ and }m=\frac 74(1-\theta).
\]
Thus, in particular $\tilde v\in L^2(0,T;H^{7/4}(\Gamma_L))$. In addition, from the above considerations, we have that $\tilde v\in H^{3/2}(0,T;L^2(\Gamma_L))$. We then conclude that $\tilde v\in H^{7/4}(\Gamma_L\times(0,T))$. 
\end{proof}

Given $(\tilde v,\tilde p)\in \mathcal E$, we look for a solution $(u,v,p)$ to the following (decoupled) problems: 
\begin{description}[nolistsep,leftmargin=0pt]
\item[{\bf The elastic problem.}]
\begin{equation}\label{eq:elastic}
\begin{aligned}
&\rho_B\frac{\partial^2 u}{\partial t^2}=\div P(u) &&\text{in }\Omega_B\times(0,T),
\\
&P(u)\cdot \n=0 &&\text{on }\Gamma_B\times(0,T),
\\
&\frac{\partial u}{\partial t}=\tilde v &&\text{on }\Gamma_L\times(0,T),
\\
&u(\cdot,0)=u_0 &&\text{in }\Omega_B,
\\
&\frac{\partial u}{\partial t}(\cdot,0)=u_1 &&\text{in }\Omega_B,
\end{aligned}
\end{equation}
with the first Piola-Kirchhoff stress $P$ defined in \eqref{eq:piola}. 
\item[{\bf The fluid problem.}]
\begin{equation}\label{eq:fluid}
\begin{aligned}
&\div v=g(\tilde v,\hat \chi)&&\text{in }\Omega_L\times (0,T),
\\
&\rho_{_{L}}\frac{\partial v}{\partial t}=\div \mathcal T(v,p)+f(\tilde v,\tilde p,\hat \chi) &&\text{in }\Omega_L\times (0,T),
\\
&\mathcal T(v,p)\cdot \n=d(u,\tilde v,\tilde p,\hat \chi)&&\text{on }\Gamma_L\times(0,T),
\\
&v(\cdot,0)=V_0(\cdot,0) &&\text{in }\Omega_L,
\end{aligned}
\end{equation}
where the Cauchy stress tensor $\mathcal T$ is the one defined in \eqref{eq:cauchy} and the forcing terms are given by 
\begin{align}
&g(\tilde v,\hat \chi):=\div \tilde v-\cof{\hat F}:\nabla \tilde v 
=(I-\cof{\hat F}):\nabla \tilde v,\label{eq:g}
\\
& \mathcal T_{\hat \chi}(\tilde v,\tilde p)=-\tilde p\cof{\hat F}+\mu\nabla \tilde v\cdot(\cof{\hat F})^T\cdot\cof{\hat F}+\mu\cof{\hat F}\cdot(\nabla \tilde v)^T\cdot\cof{\hat F}\nonumber
\\
&f(\tilde v,\tilde p,\hat \chi):=\div \mathcal T_{\hat\chi}(\tilde v,\tilde p)-\div \mathcal T(\tilde v,\tilde p)\label{eq:f}
\\
&\qquad\qquad = (I-\cof{\hat F})\cdot\nabla \tilde p+\mu\frac{\partial}{\partial X_\beta}\left[(\cof{\hat F}_{\alpha\beta}\cof{\hat F}_{\alpha\gamma}-\delta_{\alpha\beta}\delta_{\alpha\gamma})\frac{\partial \tilde v}{\partial X_\gamma}\right]\nonumber
\\
&\qquad \qquad \qquad \qquad \qquad \qquad \quad+\mu\frac{\partial}{\partial X_j}\left[ (\cof{\hat F}_{ik}\cof{\hat F}_{\ell j}-\delta_{ik}\delta_{\ell j})\frac{\partial \tilde v_\ell}{\partial X_k}\right]e_i, \nonumber
\\
&d(u,\tilde v,\tilde p,\hat \chi)=[\mathcal T(\tilde v,\tilde p)-\mathcal T_{\hat \chi}(\tilde v,\tilde p)+P(u)]\cdot \n\label{eq:d}
\\
&\qquad\qquad=-\tilde p(I-\cof{\hat F})\cdot \n\nonumber
\\
&\qquad\qquad\quad+\mu\left[(\delta_{\gamma \delta}\delta_{\gamma \beta}-\cof{\hat F}_{\gamma \delta}\cof{\hat F}_{\gamma\beta})\frac{\partial \tilde v_{\alpha}}{\partial X_\delta} \right.
\\
&\qquad\qquad\qquad\qquad\qquad \left. +
(\delta_{\alpha \iota}\delta_{\kappa\beta}-\cof{\hat F}_{\alpha\iota}\cof{\hat F}_{\kappa\beta})\frac{\partial \tilde v_{\kappa}}{\partial X_\iota}\right]\n_\beta e_\alpha 
+P(u)\cdot \n.\nonumber
\end{align}
\item[{\bf The flow map problem.}]
\begin{equation}\label{eq:flow}
\begin{aligned}
&\frac{d \chi}{d t}=v &&\text{in }\Omega_L\times (0,T),
\\
&\chi(\cdot,0)=I&&\text{in }\Omega_L. 
\end{aligned}
\end{equation}
\end{description}
In the above systems, the initial conditions satisfy the compatibility conditions \emph{(i)--(iii)} in Theorem \ref{th:existence-eom-L-c-hat}. 
Once the existence of $(u,v,p)\in \mathcal U\times \mathcal V_2\times \Pi_2$ solution of \eqref{eq:elastic}--\eqref{eq:fluid} 
 is proved, we consider the map:
\[\begin{split}
\mathcal F:\; (\tilde v,\tilde p)\in \mathcal E\mapsto \mathcal F(\tilde v,\tilde p)=(v,p)&\text{ solution to \eqref{eq:fluid}}
\\
&\text{ (corresponding to the data $(u_0,u_1,V_0)$),}
\end{split}\]
and prove that $\mathcal F$ admits a fixed point $(\bar v,\bar p)$. Let $\bar u$ be the corresponding solution of \eqref{eq:elastic} with $\tilde v=\bar v$, then the triple $(\bar u,\bar v,\bar p)$ solves \eqref{eq:eom-L-c-hat}. 

\begin{remark}\label{rm:fixedpoint0data}
For the proof of the fixed point, we will use Banach fixed point theorem. We notice that $\mathcal E$ is a complete metric space with the metric induced by the norm $\norm{\cdot}_{\mathcal V_2}+\norm{\cdot}_{\Pi_2}$ on $\mathcal V_2\times \Pi_2$. To prove that $\mathcal F$ is a contraction, i.e., 
\[
\norm{\mathcal F(\tilde v_1,\tilde p_1)-\mathcal F(\tilde v_1,\tilde p_1)}_{\mathcal V_2\times \Pi_2}\le \alpha\norm{(\tilde v_1,\tilde p_1)-(\tilde v_2,\tilde p_2)}_{\mathcal V_2\times \Pi_2},
\]
for some $\alpha\in [0,1)$, it is enough to show that the linear map  
\[\begin{split}
{}_0\mathcal F:\; (\tilde v,\tilde p)\in {}_0\mathcal E
\mapsto \;&{}_0\mathcal F(\tilde v,\tilde p)=(v,p)\text{ solution to \eqref{eq:fluid} }
\\
&\text{(corresponding to the data $(u_0=0,u_1=0,V_0=0)$)}
\end{split}\]
satisfies the following estimate for some $\beta\in [0,1)$: 
\[
\norm{ {}_0\mathcal F(w,\pi)}_{\mathcal V_2\times\Pi_2}\le\beta \norm{ (w,\pi)}_{\mathcal V_2\times\Pi_2}
\]  
for all $(w,\pi)\in {}_0\mathcal E$. In the latter displayed equation   
\begin{equation}\label{eq:B_M_0data}
{}_0\mathcal E:=\left\{(\tilde v,\tilde p)\in \mathcal E:\; \tilde v(\cdot,0)=\frac{\partial \tilde v}{\partial t}(\cdot,0)=\tilde p(\cdot,0)=0\right\}.
\end{equation}
Our claim follows from the fact that equations \eqref{eq:elastic} and \eqref{eq:fluid} are linear, non-homogeneous partial differential equations with forcing terms that are linear in their arguments. Thus, as long as solutions to \eqref{eq:elastic} and \eqref{eq:fluid} exist, the couple 
\[
(w,\pi):=\mathcal F(\tilde v_1,\tilde p_1)-\mathcal F(\tilde v_1,\tilde p_1)
\]
can be considered as solution to \eqref{eq:fluid} with $(\tilde v,\tilde p)=(\tilde v_1,\tilde p_1)-(\tilde v_2,\tilde p_2)$ and zero initial data. 
\end{remark}

For the existence of solutions to the problems \eqref{eq:elastic} and \eqref{eq:fluid} we will use Theorems \ref{th:N} and \ref{th:nStokes_r}, respectively, provided that the data satisfy the regularity assumptions of the two theorems. This will be done next. 

\begin{proposition}[Well-posedness of the elastic problem \eqref{eq:elastic}]\label{prop:elastic}
Let $T\in (0,1)$, $(\tilde v,\tilde p)\in \mathcal E$ with $\mathcal E$ defined in \eqref{eq:B_M}, and $(u_0,u_1)\in H^{3}(\Omega_B)\times H^{3/2}(\Omega_B)$ satisfying the compatibility conditions 
\[\begin{aligned}
&P(u_0)\cdot\n=0&&\text{ on }\Gamma_B,
\\
&u_1=\tilde v(\cdot,0)&& \text{ on }\Gamma_L,
\\
&\frac{\partial \tilde v}{\partial t}(\cdot,0)=\frac{1}{\rho_B}\div P(u_0)&&\text{on }\Gamma_L. 
\end{aligned}\]
Then, there exists a unique solution $u\in \mathcal U$ to \eqref{eq:elastic} satisfying the following estimates
\begin{equation}\label{eq:estimate-elastic}
\begin{split}
&\norm{u}_{\mathcal U}\le C\left[(1+T^{1/2})\norm{u_0}_{H^{3}(\Omega_B)}+\norm{u_1}_{H^{1+1/2}(\Omega_B)}+T^{1/8}\norm{\tilde v}_{\mathcal V_2}\right.
\\
&\qquad\qquad\qquad\qquad\qquad\qquad\qquad\qquad \left.+T^{1/8}\norm{V_0}_{H^2(\Omega_L)}+T^{1/8}\norm{V_1}_{H^1(\Omega_L)}\right],
\\
&\norm{P(u)\cdot\n}_{H^{3/2,3/2}(\Gamma_L\times(0,T))}\le C\left[T^{1/2}\norm{u_0}_{H^3(\Omega_B)}+T^{1/4}\norm{\tilde v}_{\mathcal V_2}\right.
\\
&\qquad\qquad\qquad\qquad\qquad\qquad\qquad\qquad\qquad\qquad\qquad\qquad \left.+T^{1/8}\norm{V_0}_{H^{5/2}(\Omega_L)}\right]
\\
&\norm{P(u)\cdot \n}_{H^1(0,T;H^{1/2}(\Gamma_L)}) \le C\left(T^{1/2}\norm{u}_{\mathcal U}+ T^{1/4}\norm{\tilde v}_{\mathcal V_2} +T^{1/8}\norm{V_0}_{H^2(\Omega_L)}\right),
\end{split}
\end{equation}
for some positive constant $C$ independent of $T$. 
\end{proposition} 

\begin{proof}
Given the regularity of $\tilde v$, and by Remark \ref{rm:trace-regularity-v}, it follows that 
\[
\tilde v|_{\Gamma_L}\in L^2(0,T;H^{5/2}(\Gamma_L))\cap H^{3/2}(0,T;H^{1/2}(\Gamma_L)).
\]
The proposition then follows from the estimates \eqref{eq:tildev0}, \eqref{eq:tildev5} and \eqref{eq:tildev6}, after applying Theorem \ref{th:N} and the estimates \eqref{eq:estimate-N}. 
\end{proof}

Next proposition will ascertain the well-posedness of the fluid problem. 

\begin{proposition}[Well-posedness of the fluid problem \eqref{eq:fluid}]\label{prop:fluid}
Consider $(\tilde v,\tilde p)\in \mathcal E$ with $\mathcal E$ defined in \eqref{eq:B_M}, and $u\in C([0,T]; H^{5/2}(\Omega_B))$ satisfying \eqref{eq:estimate-elastic} from Proposition \ref{prop:elastic}. Consider the functions $g$, $f$, and $d$ defined in equations \eqref{eq:g}, \eqref{eq:f}, and \eqref{eq:d}, respectively (with $\hat F=\nabla \hat \chi$, $\hat\chi\in B_R$ and $B_R$ defined in \eqref{eq:B_R}). Fix $T\in (0,1)$. 

Then, corresponding to $V_0$ (satisfying the compatibility conditions {\em (i)--(iv)} in Theorem \ref{th:existence-eom-L-c-hat}), there exists a unique solution $(v,p)\in \mathcal V_2\times \Pi_2$ to \eqref{eq:fluid}.  
In addition, 
$(v,p)$ satisfies the following estimate
\begin{equation}\label{eq:estimate-fluid}
\begin{split}
\norm{v}_{{\mathcal V}_2}+\norm{p}_{\Pi_2}\le C_T & (R+1)  \left[T^a(\norm{\tilde v}_{\mathcal V_2}+\norm{\tilde p}_{\Pi_2})
+\norm{V_0}_{H^{5/2}(\Omega_L)}+\norm{V_1}_{H^1(\Omega_L)}\right. 
\\
&\left.\phantom{T^{1/4}}+\norm{Q_0}_{H^{3/2}(\Omega_L)}+\norm{u_0}_{H^{3}(\Omega_B)}+\norm{u_1}_{H^{3/2}(\Omega_B)}
\right], 
\end{split}\end{equation}
where $C_T$ is a positive constant depending on $T$ (and non-decreasing with $T$), and $a\in (0,1)$. 
\end{proposition}

This proposition immediately follows from Theorem \ref{th:nStokes_r}, provided that the hypotheses are satisfied. With similar arguments as in Remark \ref{rm:compatibility-conditions} (see also Remark \ref{rm:compatibility-conditions-hatchi}), the initial data satisfy the compatibility conditions required by Theorem \ref{th:nStokes_r}. Next lemma states that the forcing terms are in the ``correct'' functional spaces and satisfy estimates that, once replaced in \eqref{eq:estimate_n-S}, will lead to \eqref{eq:estimate-fluid}. 

\begin{lemma}\label{lem:estimatefgd}
Under the same hypotheses of Proposition \ref{prop:fluid}, the functions $g$, $f$, and $d$ defined in equations \eqref{eq:g}, \eqref{eq:f}, and \eqref{eq:d}, respectively (and with $\hat F=\nabla \hat \chi$, $\hat\chi\in B_R$ with $B_R$ defined in \eqref{eq:B_R}), satisfy the following regularity properties
\[
\begin{split}
&g\in L^2(0,T;H^2(\Omega_L))\cap H^1(0,T;H^1(\Omega_L))\cap H^2(0,T;H^{-1}(\Omega_L)),
\\
&f\in L^2(0,T;H^1(\Omega_L))\cap H^1(0,T;L^2(\Omega_L)), 
\\
&d\in L^2(0,T;H^{3/2}(\Gamma_L))\cap H^1(0,T;H^{1/2}(\Gamma_L))\cap H^{1+1/4}(0,T;L^2(\Gamma_L)).
\end{split}
\]
In addition,   
there exist a positive constant $C$ and $a\in (0,1)$ such that for all $T>0$ satisfying $T^{1/2}R\le 1$, the following estimates hold
\begin{multline}\label{eq:estimate_g}
\norm{g}_{L^2(0,T;H^2(\Omega_L))}+\norm{\frac{\partial g}{\partial t}}_{L^2(0,T;H^{1}(\Omega_L))}+\norm{\frac{\partial^2 g}{\partial t^2}}_{L^2(0,T;H^{-1}(\Omega_L))}
\\
\qquad\qquad\qquad\qquad\qquad \le CR\left(T^{a}\norm{\tilde v}_{\mathcal V_2}+\norm{V_0}_{H^{5/2}(\Omega_L)}+\norm{V_1}_{H^1(\Omega_L)}\right),
\end{multline}
\begin{multline}\label{eq:estimate_f}
\norm{f}_{L^2(0,T;H^1(\Omega_L))}+\norm{\frac{\partial f}{\partial t}}_{L^2(0,T;L^2(\Omega_L))}
\le CR\left[T^{a}(\norm{\tilde v}_{\mathcal V_2}+\norm{\tilde p}_{\Pi_2}) \right. 
\\
\left.+\norm{V_0}_{H^{5/2}(\Omega_L)}+\norm{V_1}_{H^1(\Omega_L)} +\norm{Q_0}_{H^{3/2}(\Omega_L)}\right],
\end{multline}
\begin{align}\label{eq:estimate_d}
\norm{d}_{L^2(0,T;H^{3/2}(\Gamma_L))}&+\norm{\frac{\partial d}{\partial t}}_{L^2(0,T;H^{1/2}(\Gamma_L))}
+ \norm{d}_{H^{1/4}(0,T;L^2(\Gamma_L))}
\\
&+\norm{\frac{\partial d}{\partial t}}_{H^{1/4}(0,T;L^2(\Gamma_L))}
\le C(1+R)\left[T^{a}(\norm{\tilde v}_{\mathcal V_2}+\norm{\tilde p}_{\Pi_2})\right. \notag
\\
&\left.\phantom{T^{1/4}} +\norm{V_0}_{H^{5/2}(\Omega_L)}+\norm{V_1}_{H^1(\Omega_L)} 
+\norm{u_0}_{H^{3}(\Omega_B)}+\norm{u_1}_{H^{3/2}(\Omega_B)}
\right].\notag
\end{align}
\end{lemma} 

\begin{proof}
We will prove the regularity properties of $g$, $f$ and $d$ by directly proving \eqref{eq:estimate_g}, \eqref{eq:estimate_f}, and \eqref{eq:estimate_d}. Let us start with \eqref{eq:estimate_g}. We notice that 
\[\begin{split}
\norm{g}_{L^2(0,T;H^2(\Omega_L))}^2&=\int^T_0\norm{(I-\cof{\hat F}):\nabla \tilde v}^2_{L^2(\Omega_L)}
\\
&\qquad +\int^T_0\norm{\nabla[(I-\cof{\hat F}):\nabla \tilde v]}^2_{L^2(\Omega_L)}
\\
&\qquad +\int^T_0\norm{D^2[(I-\cof{\hat F}):\nabla \tilde v]}^2_{L^2(\Omega_L)}.
\end{split}\]
Each of the integrals on the right-hand side of the displayed equation can be estimated as follows with the aid of H\"older's  inequality and Sobolev embeddings
\begin{equation}\label{eq:eg1}
\int^T_0\norm{(I-\cof{\hat F}):\nabla \tilde v}^2_{L^2(\Omega_L)}\le \norm{I-\cof{\hat F}}^2_{C([0,T]\times \overline{\Omega_L})}\norm{\nabla\tilde v}^2_{L^2(0,T;L^2(\Omega_L))},
\end{equation}
\begin{equation}\label{eq:eg2}\begin{split}
&\int^T_0\norm{\nabla[(I-\cof{\hat F}):\nabla \tilde v]}^2_{L^2(\Omega_L)}
\\
&\ 
\le c_1\left[\norm{\nabla \cof{\hat F}}^2_{C([0,T];L^3(\Omega_L))}+\norm{I-\cof{\hat F}}^2_{C([0,T]\times\overline{\Omega_L})}\right]\norm{\tilde v}^2_{L^2(0,T;H^2(\Omega_L))},
\end{split}\end{equation}
\begin{equation}\label{eq:g3}\begin{split}
&\int^T_0\norm{D^2[(I-\cof{\hat F}):\nabla \tilde v]}^2_{L^2(\Omega_L)}
\le c_2\left[
\norm{D^2\cof{\hat F}}^2_{C([0,T];L^2(\Omega_L))}\right.
\\
&\quad \left.
+\norm{I-\cof{\hat F}}^2_{C([0,T]\times\overline{\Omega_L})}
+\norm{\nabla \cof{\hat F}}^2_{C([0,T];L^3(\Omega_L))}
\right]\norm{\tilde v}^2_{L^2(0,T;H^3(\Omega_L))}. 
\end{split}\end{equation}
Let us look at the second term in the left-hand side of \eqref{eq:estimate_g}: 
\[\begin{split}
\norm{\frac{\partial g}{\partial t}}_{L^2(0,T;H^1(\Omega_L))}^2
&\le2\left[\int^T_0\norm{\frac{\partial }{\partial t}(I-\cof{\hat F}):\nabla \tilde v}^2_{L^2(\Omega_L)}\right.
\\
&\qquad +\int^T_0\norm{(I-\cof{\hat F}):\nabla \left(\frac{\partial \tilde v}{\partial t}\right)}^2_{L^2(\Omega_L)}
\\
&\qquad+\int^T_0\norm{\nabla[\frac{\partial}{\partial t}(I-\cof{\hat F}):\nabla \tilde v]}^2_{L^2(\Omega_L)}
\\
&\qquad\left.+\int^T_0\norm{\nabla[(I-\cof{\hat F}):\nabla \left(\frac{\partial \tilde v}{\partial t}\right)]}^2_{L^2(\Omega_L)}\right].
\end{split}\]
For the second and fourth term on the right-hand side, thanks to the regularity properties of $\tilde v$, we can proceed exactly like in the estimates \eqref{eq:eg1} and \eqref{eq:eg2}, thus obtaining
\begin{multline}\label{eq:g4}
\int^T_0\norm{(I-\cof{\hat F}):\nabla \left(\frac{\partial \tilde v}{\partial t}\right)}^2_{L^2(\Omega_L)}
\le \norm{I-\cof{\hat F}}^2_{C([0,T]\times \overline{\Omega_L})}\norm{\frac{\partial \tilde v}{\partial t}}^2_{L^2(0,T;H^1(\Omega_L))},
\end{multline}
and 
\begin{equation}\label{eq:g5}\begin{split}
&\int^T_0\norm{\nabla[(I-\cof{\hat F}):\nabla \left(\frac{\partial \tilde v}{\partial t}\right)]}^2_{L^2(\Omega_L)}
\le c_3\left[\norm{\nabla \cof{\hat F}}^2_{C([0,T];L^3(\Omega_L))} \right.
\\
&\qquad\qquad\qquad\qquad\qquad\qquad\quad \left.+\norm{I-\cof{\hat F}}^2_{C([0,T]\times\overline{\Omega_L})}\right]\norm{\frac{\partial\tilde v}{\partial t}}^2_{L^2(0,T;H^2(\Omega_L))}.
\end{split}\end{equation}
We then estimate the remaining integrals by repeatedly using H\"older's  inequality and Sobolev embedding, obtaining the following estimates:
\begin{equation}\label{eq:g6}\begin{split}
&\int^T_0\norm{\frac{\partial }{\partial t}(I-\cof{\hat F}):\nabla \tilde v}^2_{L^2(\Omega_L)}\le 
c_4\int^T_0\norm{\frac{\partial }{\partial t}\cof{\hat F}}^2_{L^3(\Omega_L)}\norm{\tilde v}^2_{H^2(\Omega_L)}
\\
&\qquad\qquad\qquad\qquad\qquad\le c_5\norm{\frac{\partial }{\partial t}\cof{\hat F}}^2_{L^\infty(0,T;H^1(\Omega_L))}\norm{\tilde v}^2_{C([0,T];H^2(\Omega_L))}T;
\end{split}\end{equation}
\begin{equation}\label{eq:g7}\begin{split}
&\int^T_0\norm{\nabla[\frac{\partial}{\partial t}(I-\cof{\hat F}):\nabla \tilde v]}^2_{L^2(\Omega_L)}
\le c_6 \left[\norm{\frac{\partial}{\partial t}\cof{\hat F}}^2_{L^\infty(0,T;H^1(\Omega_L))}T \right.
\\
&\qquad\qquad\qquad\qquad\qquad\qquad \left.+\norm{\frac{\partial }{\partial t}\nabla \cof{\hat F}}^2_{L^2(0,T;L^3(\Omega_L))}\right]\norm{\tilde v}^2_{C([0,T];H^2(\Omega_L))}. 
\end{split}\end{equation}
Let us now estimate 
\[\begin{split}
\norm{\frac{\partial^2 g}{\partial t^2}}_{L^2(0,T;H^{-1}(\Omega_L))}^2
&=\int^T_0\norm{\frac{\partial^2}{\partial t^2}\left[(I-\cof{\hat F}):\nabla \tilde v\right]}^2_{H^{-1}(\Omega_L)}
\\
&\le 2\int^T_0\norm{\frac{\partial^2 }{\partial t^2}(I-\cof{\hat F}):\nabla \tilde v}^2_{H^{-1}(\Omega_L)}
\\
&\quad +4\int^T_0\norm{\frac{\partial}{\partial t}(I-\cof{\hat F}):\nabla \left(\frac{\partial \tilde v}{\partial t}\right)}^2_{H^{-1}(\Omega_L)}
\\
&\quad +2\int^T_0\norm{(I-\cof{\hat F}):\nabla \left(\frac{\partial^2 \tilde v}{\partial t^2}\right)}^2_{H^{-1}(\Omega_L)}. 
\end{split}\]
For the first integral on the right-hand side, we notice that $L^{6/5}(\Omega_L)\hookrightarrow H^{-1}(\Omega_L)$, then H\"older's inequality and Sobolev embedding theorem imply the following estimate 
\begin{multline}\label{eq:g8}
\int^T_0\norm{\frac{\partial^2}{\partial t^2}(I-\cof{\hat F}):\nabla \tilde v}^2_{H^{-1}(\Omega_L)}
\\
\le c_7\norm{\frac{\partial^2 }{\partial t^2}\cof{\hat F}}^2_{L^2(0,T;L^2(\Omega_L))}\norm{\tilde v}^2_{C([0,T];H^2(\Omega_L))}.
\end{multline}
For the second integral, using a similar strategy, we find 
\begin{multline}\label{eq:g9}
\int^T_0\norm{\frac{\partial}{\partial t}(I-\cof{\hat F}):\nabla \left(\frac{\partial \tilde v}{\partial t}\right)}^2_{H^{-1}(\Omega_L)}
\\
\le c_8\norm{\frac{\partial}{\partial t}\cof{\hat F}}^2_{L^\infty(0,T;H^1(\Omega_L))}\norm{\nabla \left(\frac{\partial \tilde v}{\partial t}\right)}^2_{C([0,T];L^2(\Omega_L))} T. 
\end{multline} 
For the third integral, we first notice that 
\[
(I-\cof{\hat F}):\nabla \left(\frac{\partial^2 \tilde v}{\partial t^2}\right)=\frac{\partial}{\partial X_\beta}\left[(\delta_{\alpha \beta}-\cof{\hat F}_{\alpha\beta})\frac{\partial^2 \tilde v_\alpha}{\partial t^2}\right]-\frac{\partial \cof{\hat F}_{\alpha\beta}}{\partial X_\beta}\frac{\partial^2 \tilde v_\alpha}{\partial t^2},
\]
and since (by \eqref{eq:piola-condition}, c.f. \eqref{eq:B_R})
\[
\frac{\partial \cof{\hat F}_{\alpha\beta}}{\partial X_\beta}=0\qquad\text{ for all }\alpha=1,2,3,
\]
then 
\[
\norm{(I-\cof{\hat F}):\nabla \left(\frac{\partial^2 \tilde v}{\partial t^2}\right)}_{H^{-1}(\Omega_L)}
\le c_9\norm{(I-\cof{\hat F})^T\cdot\frac{\partial^2 \tilde v}{\partial t^2}}_{L^{2}(\Omega_L)}. 
\]
Thus, we have the estimate 
\begin{multline}\label{eq:g10}
\int^T_0\norm{(I-\cof{\hat F}):\nabla \left(\frac{\partial^2 \tilde v}{\partial t^2}\right)}^2_{H^{-1}(\Omega_L)}
\\
\le c_{10}\norm{I-\cof{\hat F}}_{C([0,T]\times\overline{\Omega_L})}^2
\norm{\frac{\partial^2 \tilde v}{\partial t^2}}_{L^2(0,T;L^2(\Omega_L))}^2. 
\end{multline}

The proof of \eqref{eq:estimate_g} then follows by using Lemma \ref{lem:estimate_hatchi} and Lemma \ref{lem:r-tilde-vq} in the right-hand sides of \eqref{eq:eg1}--\eqref{eq:g10}. 

Let us now  introduce the tensor-valued function
\begin{equation}\label{eq:def-S}\begin{split}
S(\tilde v,\hat \chi)&:=2D(\tilde v)-\nabla \tilde v\cdot\cof{\hat F}^T\cdot\cof{\hat F}-\cof{\hat F}\cdot(\nabla\tilde v)^T\cdot\cof{\hat F}
\\
&=\nabla{\tilde v}-\nabla{\tilde v}\cdot\cof{\hat F}^T\cdot\cof{\hat F}+(\nabla{\tilde v})^T-\cof{\hat F}\cdot(\nabla\tilde v)^T\cdot\cof{\hat F}
\\
&=\nabla{\tilde v}\cdot(I-\cof{\hat F}^T)+\nabla{\tilde v}\cdot\cof{\hat F}^T\cdot(I-\cof{\hat F})
\\
&\qquad+(I-\cof{\hat F})\cdot (\nabla{\tilde v})^T+\cof{\hat F}\cdot(\nabla\tilde v)^T\cdot(I-\cof{\hat F})
\\
&=\nabla{\tilde v}\cdot(I-\cof{\hat F}^T)+[(\nabla{\tilde v})\cdot (I-\cof{\hat F}^T)]^T
\\
&\qquad+\nabla{\tilde v}\cdot\cof{\hat F}^T\cdot(I-\cof{\hat F})+[(\nabla\tilde v)\cdot\cof{\hat F}^T]^T\cdot(I-\cof{\hat F}).
\end{split}\end{equation}
From \eqref{eq:f}, we can then write  $f= (I-\cof{\hat F})\cdot\nabla \tilde p-\mu \div S$. For the proof of \eqref{eq:estimate_f}, we first notice that, up to an absolute constant $c$ (depending at most on $\mu$),
\[\begin{split}
&|f|\le c\left(|I-\cof{\hat F}||\nabla \tilde p|+|D^2_X\tilde v| |I-\cof{\hat F}|
+|\nabla\tilde v||\nabla\cof{\hat F}| \right.
\\
&\qquad\qquad\qquad+ |D^2_X\tilde v| |\cof{\hat F}| |I-\cof{\hat F}| + |\nabla\tilde v||\nabla\cof{\hat F}| |I-\cof{\hat F}|
\\
&\qquad\qquad\qquad\left.+|\nabla\tilde v| |\cof{\hat F}| |\nabla\cof{\hat F}|\right),
\\
&|\frac{\partial f}{\partial X_\iota}|\le  c\left(|\nabla\cof{\hat F}| |\nabla \tilde p| +  |I-\cof{\hat F}| |D^2_X \tilde p| \right.
\\
&\qquad\qquad\qquad +|D^3_X\tilde v| |I-\cof{\hat F}| + |D^2_X\tilde v||\nabla\cof{\hat F}| + |\nabla\tilde v| |D^2_X\cof{\hat F}|
\\
&\qquad\qquad\qquad +|D^3_X\tilde v||\cof{\hat F}| |I-\cof{\hat F}| 
+|D^2_X\tilde v||\nabla\cof{\hat F}| |I-\cof{\hat F}|
\\
&\qquad\qquad\qquad  +|D^2_X\tilde v||\cof{\hat F}||\nabla\cof{\hat F}|+|\nabla\tilde v| |D^2_X\cof{\hat F}| |I-\cof{\hat F}| 
\\
&\qquad\qquad\qquad  \left. 
+|\nabla\tilde v||\nabla\cof{\hat F}|^2
+|\nabla\tilde v||\cof{\hat F}| |D^2_X\cof{\hat F}|\right), 
\\
&|\frac{\partial f}{\partial t}|\le c\left(|\frac{\partial \cof{\hat F}}{\partial t}||\nabla \tilde p|+|I-\cof{\hat F}||\frac{\partial }{\partial t}\nabla\tilde p| 
+|\frac{\partial }{\partial t}D^2_X\tilde v| |I-\cof{\hat F}| \right. 
\\
&\qquad\qquad \quad  + |D^2_X\tilde v||\frac{\partial}{\partial t}\cof{\hat F}| 
+|\frac{\partial }{\partial t}\nabla\tilde v||\nabla\cof{\hat F}| 
+|\nabla\tilde v||\frac{\partial}{\partial t}\nabla\cof{\hat F}|
\\
&\qquad\qquad\quad 
+|\frac{\partial}{\partial t}D^2_X\tilde v||\cof{\hat F}| |I-\cof{\hat F}|
+|D^2_X\tilde v||\frac{\partial}{\partial t}\cof{\hat F}| |I-\cof{\hat F}|
\\
&\qquad\qquad\quad 
+|D^2_X\tilde v||\cof{\hat F}||\frac{\partial }{\partial t}\cof{\hat F}|
+|\frac{\partial}{\partial t}\nabla\tilde v|\nabla\cof{\hat F}| |I-\cof{\hat F}|
\\
&\qquad\qquad\quad 
+|\nabla\tilde v||\frac{\partial}{\partial t}\nabla\cof{\hat F}| |I-\cof{\hat F}|+|\nabla\tilde v||\nabla\cof{\hat F}| |\frac{\partial}{\partial t}\cof{\hat F}|
\\
&\qquad\qquad\quad \left. 
+|\frac{\partial}{\partial t}\nabla\tilde v||\cof{\hat F}||\nabla\cof{\hat F}| + |\nabla\tilde v||\cof{\hat F}||\frac{\partial}{\partial t}\nabla\cof{\hat F}|\right). 
\end{split}\]
We will only estimate the terms involving the pressure field $\tilde p$ on the right-hand side of the above inequalities; all the remaining terms can be estimated in a similar fashion as in the previous and in the following calculations, using Lemma \ref{lem:estimate_hatchi} and Lemma \ref{lem:r-tilde-vq}. By \eqref{eq:iv2}
\[\left.\begin{split}
&\int^T_0\int_{\Omega_L} | I-\cof{\hat F}|^2|\nabla \tilde p|^2
\\
&\int^T_0\int_{\Omega_L}|I-\cof{\hat F}|^2 |D^2 \tilde p|^2
\\
&\int^T_0\int_{\Omega_L} |I-\cof{\hat F}|^2|\frac{\partial }{\partial t}\nabla\tilde p|^2
\end{split}\right\}
\le CTR^2\norm{\tilde p}_{\Pi_2}^2. 
\]
By H\"older's inequality, Sobolev embedding, and by \eqref{eq:iv3}, we obtain
\[\begin{split}
\int^T_0\int_{\Omega_L}|\nabla\cof{\hat F}|^2 |\nabla \tilde p|^2&\le\int^T_0 \norm{\nabla\cof{\hat F}}_{L^3(\Omega_L)}^{2}\norm{\nabla \tilde p}_{L^6(\Omega_L)}^{2}
\\
&\le c_{11}R^{2}T^{4/3}\norm{\tilde p}_{\Pi_2}^{2}. 
\end{split}\]
Similarly, using \eqref{eq:iv6} and \eqref{eq:tildep1}, we find that 
\[\begin{split}
\int^T_0\int_{\Omega_L}|\frac{\partial \cof{\hat F}}{\partial t}|^2 |\nabla \tilde p|^2&\le \int^T_0\norm{\frac{\partial \cof{\hat F}}{\partial t}}_{L^4(\Omega_L)}^{2} \norm{\nabla \tilde p}_{L^4(\Omega_L)}^2
\\
&\le c_{12}R^2T\left(\norm{\tilde p}_{\Pi_2}^2+\norm{Q_0}_{H^{3/2}(\Omega_L)}\right). 
\end{split}\]
We now prove \eqref{eq:estimate_d}. From the definition of $d$ in \eqref{eq:d} and from \eqref{eq:def-S}, 
we write $d=[-\tilde p(I-\cof{\hat F})+\mu S(\tilde v,\hat \chi)+P(u)]\cdot\n$, and \eqref{eq:estimate_d} would follow if we could prove that $S$ satisfies the following estimate
\begin{equation}\label{eq:estimate-S}
\begin{split}
\norm{S}_{L^2(0,T;H^{3/2}(\Gamma_L))}&+\norm{\frac{\partial S}{\partial t}}_{L^2(0,T;H^{1/2}(\Gamma_L))}
+ \norm{S}_{H^{1/4}(0,T;L^2(\Gamma_L))}
\\
&+\norm{\frac{\partial S}{\partial t}}_{H^{1/4}(0,T;L^2(\Gamma_L))}\le CR\left[T^{\alpha}(\norm{\tilde v}_{\mathcal V_2}+\norm{\tilde p}_{\Pi_2})\right.
\\
&\qquad\qquad\qquad\qquad\qquad\qquad\left.+\norm{V_0}_{H^{5/2}(\Omega_L)}+\norm{V_1}_{H^1(\Omega_L)} \right].
\end{split}\end{equation}
for some $\alpha \in (0,1)$ and a positive constant $C$ independent of $T$. In fact, given that $u$ satisfies the estimates \eqref{eq:estimate-elastic} (recall \eqref{eq:normU}),  
 $(\tilde v,\tilde p)\in \mathcal E$ (recall \eqref{eq:B_M}), and \eqref{eq:iv2} holds, then the terms involving $P(u)\cdot\n$ and $\tilde p$ in 
 \eqref{eq:d} can be easily estimated as in \eqref{eq:estimate_d}.  
We estimate only $\nabla{\tilde v}\cdot\cof{\hat F}^T\cdot(I-\cof{\hat F})$ as a sample term in $S$, all the others can be estimated similarly thanks to Lemma \ref{lem:estimate_hatchi} and Lemma \ref{lem:r-tilde-vq}. We start by showing that 
\begin{equation}\label{eq:sample-S-1}
\norm{\nabla{\tilde v}\cdot\cof{\hat F}^T\cdot(I-\cof{\hat F})}_{L^2(0,T;H^2(\Omega_L)}\le CT^{1/2}R \norm{\tilde v}_{\mathcal V_2}.
\end{equation}
The estimate for the corresponding term in \eqref{eq:estimate-S} (the first one, in fact) will follow from the trace theorem. By H\"older's inequality, Sobolev embedding theorem,  \eqref{eq:iv2}, and \eqref{eq:iv1}, we immediately see that 
\begin{equation}\label{eq:sample-S-1.1}
\begin{split}
\norm{\nabla{\tilde v}\cdot\cof{\hat F}^T\cdot(I-\cof{\hat F})}_{L^2(0,T;L^2(\Omega_L))}^2
&\le k_1TR^2\int^T_0\norm{\nabla \tilde v}^2_{L^6(\Omega_L)}\norm{\cof{\hat F}}^2_{L^3(\Omega_L)}
\\
&
\le k_2TR^2\norm{\tilde v}^2_{\mathcal V_2}. 
\end{split}\end{equation}
Subsequently, we notice that  
\[\begin{split}
&\norm{\nabla{\tilde v}\cdot\cof{\hat F}^T\cdot(I-\cof{\hat F})}_{L^2(0,T;H^1(\Omega_L))}^2
\\
&\quad 
\le\norm{\nabla{\tilde v}\cdot\cof{\hat F}^T\cdot(I-\cof{\hat F})}_{L^2(0,T;L^2(\Omega_L))}^2
\\
&\quad \quad +k_3\int^T_0\int_{\Omega_L}\left[|D^2_X\tilde v|^2|\cof{\hat F}|^2|I-\cof{\hat F}|^2 
\right.
\\
&\qquad \qquad \quad \left. +|\nabla \tilde v|^2|\nabla \cof{\hat F}|^2|I-\cof{\hat F}|^2
+|\nabla \tilde v|^2|\cof{\hat F}|^2|\nabla \cof{\hat F}|^2\right].
\end{split}\]
The first two terms inside the latter displayed integral can be estimated exactly as in \eqref{eq:sample-S-1.1}. For the last term, we use \eqref{eq:iv1} for the uniform estimate of $|\cof{\hat F}|^2$, and then 
\begin{equation}\label{eq:sample-S-1.2}
\begin{split}
\int^T_0\int_{\Omega_L}|\nabla \tilde v|^2|\cof{\hat F}|^2|\nabla \cof{\hat F}|^2&\le k_4\int^T_0\norm{\nabla \tilde v}_{L^6(\Omega_L)}^2\norm{\nabla \cof{\hat F}}^2_{L^3(\Omega_L)}
\\
&\le k_5 T^{4/3}R^2\norm{\tilde v}^2_{\mathcal V_2},
\end{split}\end{equation}
by H\"older's inequality, Sobolev embedding, and \eqref{eq:iv3}. The second-order derivatives (with respect to the space variable) can be estimated similarly. For what concerns the time-derivative, 
we will show that 
\[
\norm{\frac{\partial}{\partial t}\left(\nabla{\tilde v}\cdot\cof{\hat F}^T\cdot(I-\cof{\hat F})\right)}_{L^2(0,T;H^1(\Omega_L))}\le CT^{\alpha}R \norm{\tilde v}_{\mathcal V_2}
\]
for some $\alpha \in (0,1)$ and a positive constant $C$ independent of $T$. The estimate for the corresponding term in \eqref{eq:estimate-S} (the second one on the left-hand side) will follow from the trace theorem. We first notice that 
\begin{equation*}\begin{split}
&\norm{\frac{\partial}{\partial t}\left(\nabla{\tilde v}\cdot\cof{\hat F}^T\cdot(I-\cof{\hat F})\right)}^2_{L^2(0,T;L^2(\Omega_L))}
\\
&\qquad\qquad \le k_6\left(\int^T_0\int_{\Omega_L}|\nabla\frac{\partial \tilde v}{\partial t}|^2 |\cof{\hat F}|^2 |I-\cof{\hat F}|^2\right. 
\\
&\qquad\qquad\qquad\qquad\qquad+\int^T_0\int_{\Omega_L}|\nabla \tilde v|^2 |\frac{\partial}{\partial t}\cof{\hat F}|^2 |I-\cof{\hat F}|^2
\\
&\qquad\qquad\qquad\qquad\qquad  \left.+\int^T_0\int_{\Omega_L}|\nabla\tilde v|^2 |\cof{\hat F}|^2 |\frac{\partial }{\partial t}\cof{\hat F}|^2\right). 
\end{split}\end{equation*}
The first two integral terms on the right-hand side can be estimated exactly as in \eqref{eq:sample-S-1.1} thanks to \eqref{eq:iv2}. For the last one, we use \eqref{eq:iv1} for the uniform estimate of $\cof{\hat F}$, and then H\"older's inequality with Sobolev embedding, \eqref{eq:iv6} and \eqref{eq:tildev0}, to find the following estimate 
\begin{equation}\label{eq:sample-S-2.1}\begin{split}
\int^T_0\int_{\Omega_L}|\nabla\tilde v|^2 |\cof{\hat F}|^2 |\frac{\partial}{\partial t}\cof{\hat F}|^2&\le k_7\int^T_0\norm{\nabla \tilde v}^2_{L^6(\Omega_L)}\norm{\frac{\partial}{\partial t}\cof{\hat F}}^2_{L^3(\Omega_L)}
\\
&
\le k_8R^2\left(T\norm{\tilde v}_{\mathcal V_2}^2+\norm{V_0}_{H^2(\Omega_L)}^2\right). 
\end{split}\end{equation}
Now, we notice that 
\[
|D_X\frac{\partial}{\partial t}\left(\nabla{\tilde v}\cdot\cof{\hat F}^T\cdot(I-\cof{\hat F})\right)|\le 
c\sum^8_{i=1}J_i,
\]
where
\[\begin{split}
&\sum^8_{i=1}J_i:=
|D^2_X\frac{\partial \tilde v}{\partial t}| |\cof{\hat F}| |I-\cof{\hat F}|+|\nabla\frac{\partial \tilde v}{\partial t}| |\nabla \cof{\hat F}| |I-\cof{\hat F}|
\\
&\qquad \quad +|\nabla\frac{\partial \tilde v}{\partial t}| |\cof{\hat F}| |\nabla\cof{\hat F}|
+|D^2_X \tilde v| |\frac{\partial }{\partial t}\cof{\hat F}| |I-\cof{\hat F}| 
\\
&\qquad\quad + |\nabla \tilde v| |\nabla \frac{\partial }{\partial t}\cof{\hat F}| |I-\cof{\hat F}| 
+ |\nabla \tilde v| |\frac{\partial }{\partial t}\cof{\hat F}| |\nabla \cof{\hat F}| 
\\
&\qquad \quad
+ |D^2_X\tilde v| |\cof{\hat F}| |\frac{\partial }{\partial t}\cof{\hat F}|  
+ |\nabla\tilde v| |\cof{\hat F}| |\nabla \frac{\partial }{\partial t}\cof{\hat F}|.
\end{split}\]
The terms $J_1$ and $J_2$ can be estimated as in \eqref{eq:sample-S-1.1} thanks to Lemma \ref{lem:estimate_hatchi}. $J_3$ can be estimated as \eqref{eq:sample-S-1.2}, 
whereas $J_8$ like \eqref{eq:sample-S-2.1}. 
An estimate for $J_4$ follows along the lines of \eqref{eq:sample-S-1.1} with \eqref{eq:iv1} replaced by \eqref{eq:iv6}. For $J_5$, we use \eqref{eq:iv2}, H\"older's inequality, Sobolev embedding, \eqref{eq:tildev0}, \eqref{eq:iv5} and the fact that $T<1$ and $T^{1/2}R\le1$, to obtain 
\[\begin{split}
\int^T_0\int_{\Omega_L}|\nabla \tilde v|^2 & |\nabla\frac{\partial }{\partial t}\cof{\hat F}|^2|I-\cof{\hat F}|^2
\\
\le& k_9TR^2\int^T_0\norm{\nabla \tilde v}_{L^6(\Omega_L)}^2\norm{\nabla\frac{\partial }{\partial t}\cof{\hat F}}^2_{L^3(\Omega_L)} 
\\
\le& k_{10}T^{2/3}R^2\left(\norm{\tilde v}_{\mathcal V_2}^2+\norm{V_0}_{H^{5/2}(\Omega_L)}^2\right).
\end{split}\]
Let us now estimate $J_6$. We use H\"older's inequality with exponents $p=q=r=3$, Sobolev embedding theorem, \eqref{eq:tildev0}, \eqref{eq:iv6} and \eqref{eq:iv1}, to obtain the following estimate:
\[\begin{split}
\int^T_0\int_{\Omega_L} |\nabla \tilde v|^2 & |\frac{\partial }{\partial t}\cof{\hat F}|^2  |\nabla \cof{\hat F}|^2
\\
&\le 
k_{11}\int^T_0\norm{\nabla\tilde v}^2_{L^6(\Omega_L)}\norm{\frac{\partial }{\partial t}\cof{\hat F}}^2_{L^6(\Omega_L)}\norm{\nabla \cof{\hat F}}^2_{L^6(\Omega_L)}
\\
& \le k_{12}TR^2\left(\norm{\tilde v}_{\mathcal V_2}^2+\norm{V_0}^2_{H^{5/2}(\Omega_L)}\right).
\end{split}\]
We will now estimate $J_7$. Let us apply Lemma \ref{lem:ba7} with $f=|\partial \cof{\hat F}/\partial t|$, $g=|D^2_X\tilde v|$, $\sigma=3/4$, $s=0$, $p=12$ and $q=12/5$ to obtain that 
\[\begin{split}
&\int^T_0\int_{\Omega_L}|D_X^2\tilde v|^2|\frac{\partial}{\partial t}\cof{\hat F}|^2
\\
&\ \le k_{13}\norm{\frac{\partial}{\partial t}\cof{\hat F}}^2_{L^2(0,T;L^{12}(\Omega_L))}\left[T^{1/4}\norm{D^2_X\tilde v}^2_{H^{3/4}(0,T;L^{12/5}(\Omega_L))}+\norm{V_0}_{H^{5/2}(\Omega_L)}^2\right].
\end{split}\]
From \eqref{eq:iv1}, the latter displayed inequality, Sobolev embedding theorem, \eqref{eq:iv6}, and Remark \ref{rm:trace-regularity-v}, we find the following estimate for the term $J_7$: 
\[
\int^T_0\int_{\Omega_L}|D^2_X\tilde v|^2 |\cof{\hat F}|^2 |\frac{\partial }{\partial t}\cof{\hat F}|^2
\le k_{14}T R^2\left[T^{1/4}\norm{\tilde v}_{\mathcal V_2}^2+ +\norm{V_0}_{H^{5/2}(\Omega_L)}^2\right].
\]
Putting together the above estimates, we showed that 
\begin{multline*}
\norm{\frac{\partial}{\partial t}\left(\nabla{\tilde v}\cdot\cof{\hat F}^T\cdot(I-\cof{\hat F})\right)}_{L^2(0,T;H^1(\Omega_L))}
\\
\le k_{15}R\left(T^{1/3}\norm{\tilde v}_{\mathcal V_2} + \norm{V_0}_{H^{5/2}(\Omega_L)}\right). 
\end{multline*}
To conclude the proof of \eqref{eq:estimate-S} (and thus of \eqref{eq:estimate_d}), it remains to show that the estimates
\begin{multline*}
\norm{\nabla{\tilde v}\cdot\cof{\hat F}^T\cdot(I-\cof{\hat F})}_{H^{1/4}(0,T;L^2(\Gamma_L))}
\\
\le C R\left[T^\alpha\norm{\tilde v}_{\mathcal V_2}+\norm{V_0}_{H^{5/2}(\Omega_L)}+\norm{V_1}_{H^1(\Omega_L)} \right]
\end{multline*}
and
\begin{multline*}
\norm{\frac{\partial }{\partial t}\left(\nabla{\tilde v}\cdot\cof{\hat F}^T\cdot(I-\cof{\hat F})\right)}_{H^{1/4}(0,T;L^2(\Gamma_L))}
\\
\le CR\left[T^\alpha \norm{\tilde v}_{\mathcal V_2}+\norm{V_0}_{H^{5/2}(\Omega_L)}+\norm{V_1}_{H^1(\Omega_L)} \right]
\end{multline*}
hold for some $\alpha \in (0,1)$ and a positive constant $C$ independent of $T$. We will prove only the estimate involving the time derivative because the first one follows using identical steps. 

By triangle inequality, 
\[\begin{split}
&\norm{\frac{\partial }{\partial t}\left(\nabla{\tilde v}\cdot\cof{\hat F}\cdot(I-\cof{\hat F})\right)}_{H^{1/4}(0,T;L^2(\Gamma_L))}
\\
&\qquad\le\norm{\frac{\partial }{\partial t}\left(\nabla{\tilde v}\right)\cdot \cof{\hat F}\cdot(I-\cof{\hat F})}_{H^{1/4}(0,T;L^2(\Gamma_L))}
\\
&\qquad\qquad +\norm{\nabla\tilde v\cdot\frac{\partial}{\partial t}\cof{\hat F}\cdot(I-\cof{\hat F})}_{H^{1/4}(0,T;L^2(\Gamma_L))}
\\
&\qquad\qquad +\norm{\nabla{\tilde v}\cdot\cof{\hat F}\cdot\frac{\partial }{\partial t}\cof{\hat F}}_{H^{1/4}(0,T;L^2(\Gamma_L))}.
\end{split}\]
For the first term on the right-hand side, we use Lemma \ref{lem:ba7} (with $s=1/4$, $\sigma=1$, $p=2$, $q=\infty$, and the fact that $\cof{\hat F}(\cdot, 0)=I$), \eqref{eq:tildev4.0},   
and \eqref{eq:iv4}, to obtain 
\[\begin{split}
&\norm{\frac{\partial }{\partial t}\left(\nabla{\tilde v}\right)\cdot\cof{\hat F}\cdot(I-\cof{\hat F})}_{H^{1/4}(0,T;L^2(\Gamma_L))} 
\\
&\qquad\quad\le k_{16}T^{1/4} \norm{\frac{\partial}{\partial t}\nabla\tilde v}_{H^{1/4}(0,T;L^2(\Gamma_L))} \norm{\cof{\hat F}\cdot(\cof{\hat F}-I)}_{H^1(0,T;L^\infty(\Gamma_L))}
\\
&\qquad\quad\le k_{17}T^{1/4}R\norm{\tilde v}_{\mathcal V_2}.
\end{split}\]
For the second term, we apply Lemma \ref{lem:ba7} with $p=2=4$, $\sigma=1$ and $s=1/4$, by \eqref{eq:iv7},
\begin{multline*}
\norm{\nabla\tilde v\cdot\frac{\partial}{\partial t}\cof{\hat F}\cdot(I-\cof{\hat F})}_{H^{1/4}(0,T;L^2(\Gamma_L))}
\\
\le k_{18}R\left[T^{1/4}\norm{\tilde v}_{\mathcal V_2}+ T^{1/12}\norm{V_0}_{H^2(\Omega_L)}\right]. 
\end{multline*}
For the last term, we proceed similarly to the latter. 
\end{proof}

\begin{proof}[Proof of Theorem \ref{th:existence-eom-L-c-hat}]
Fix $T\in (0,1)$. 
To prove existence and uniqueness of solutions to \eqref{eq:eom-L-c-hat}-\eqref{eq:ic-hat}, we need to show that the map 
\[\begin{split}
\mathcal F:\; (\tilde v,\tilde p)\in \mathcal E\mapsto \mathcal F(\tilde v,\tilde p)=(v,p)&\text{ solution to \eqref{eq:fluid} }
\\
&\qquad\text{(corresponding to the data $(u_0,u_1,V_0)$)}
\end{split}\]
admits a fixed point. Following Remark \ref{rm:fixedpoint0data}, we need to show that 
\begin{itemize}
\item[(I)] $\mathcal F:\; \mathcal E \to \mathcal E$; and 
\item[(II)] The map \[\begin{split}
{}_0\mathcal F:\; (\tilde v,\tilde p)\in {}_0\mathcal E&\mapsto \mathcal F(\tilde v,\tilde p)=(v,p)\text{ solution to \eqref{eq:fluid} }
\\
&\qquad\qquad\qquad \text{(corresponding to $(u_0=0,u_1=0,V_0=0)$)}
\end{split}\]
satisfies $\norm{{}_0\mathcal F(\tilde v,\tilde p)}_{\mathcal V_2\times \Pi_2}< \alpha_T \norm{(\tilde v,\tilde p)}_{\mathcal V_2\times \Pi_2}$ for some $\alpha_T\in (0,1)$,  where ${}_0\mathcal E$ is defined in \eqref{eq:B_M_0data} (see also \eqref{eq:B_M}).
\end{itemize}
Let $(\tilde v,\tilde p)\in \mathcal E$. Using Proposition \ref{prop:elastic}, we find $u\in \mathcal U$, solution to \eqref{eq:elastic}. 
Now use the same $(\tilde v,\tilde p)\in \mathcal E$ and the function $u$ just found to obtain the solution $(v,p)$ to \eqref{eq:fluid}, according to Proposition \ref{prop:fluid}. Since $(v,p)$ enjoys the following regularity
\[\begin{split}
&v\in L^2(0,T;H^3(\Omega_L))\cap H^1(0,T;H^2(\Omega_L)) \cap H^2(0,T;L^2(\Omega_L))
\\
&p\in L^2(0,T;H^2(\Omega_L))\cap H^1(0,T;H^1(\Omega_L))\cap H^{1+1/4}(0,T;L^2(\Gamma_L)),
\end{split}\]
recalling the definition of $\mathcal E$ (see \eqref{eq:B_M})), we conclude that $\mathcal F(\tilde v,\tilde p)=(v,p)\in \mathcal E$.  

It remains to show (II). 
From the estimate \eqref{eq:estimate-fluid} with zero initial data, we have that  
\[
\norm{{}_0\mathcal F(\tilde v,\tilde p)}_{\mathcal V_2\times \Pi_2}\le \alpha_T \norm{(\tilde v,\tilde p)}_{\mathcal V_2\times \Pi_2}
\]
with $\alpha_T:=C_T(R+1)T^a$, where $a\in (0,1)$ and $C_T$ is the positive constant, depending on $T$ and non-decreasing with $T$, found in Proposition \ref{prop:fluid}. We note that $\alpha_T\in (0,1)$ provided that $T$ is small enough. By Banach fixed-point theorem, we conclude that there exists a unique fixed point $(\bar v,\bar p)$. Let $\bar u$ be the corresponding solution of \eqref{eq:elastic} with $\tilde v=\bar v$, then the triple $(\bar u,\bar v,\bar p)$ solves \eqref{eq:eom-L-c-hat}.

Consider now $(u,v,p)\in \mathcal U\times\mathcal V_2\times \Pi_2$ solution to \eqref{eq:elastic}--\eqref{eq:fluid} in the maximal time interval $[0,T)$ with $T\in (0,1)$ satisfying $C_T(R+1)T^a<1$ (where $C_T$ is the positive constant, depending on $T$ and non-decreasing with $T$, found in Proposition \ref{prop:fluid}). Let us prove the energy estimate \eqref{eq:energy_hateq}. We note that the triple $(u,v,p)$ satisfies \eqref{eq:elastic} and \eqref{eq:fluid} with $(\tilde v,\tilde p)=(v,p)$. The forcing terms $g(v,p)$, $f(v,p)$ and $d(v,p)$ in \eqref{eq:fluid} still satisfy Lemma \ref{lem:estimatefgd} (with $(\tilde v,\tilde p)=(v,p)$ and non-zero initial conditions). Thus, we can apply again 
\eqref{eq:estimate-fluid} (with $(\tilde v,\tilde p)=(v,p)$ and non-zero initial conditions) to find that 
\begin{equation*}
\begin{split}
&(1-C_T (R+1)T^a )\left[\norm{v}_{{\mathcal V}_2}+\norm{p}_{\Pi_2}\right]\le C_T (R+1)  \left(\norm{V_0}_{H^{5/2}(\Omega_L)}+\norm{V_1}_{H^1(\Omega_L)}\right. 
\\
&\qquad \qquad \qquad \qquad \qquad\left.\phantom{T^{1/4}}+\norm{Q_0}_{H^{3/2}(\Omega_L)}  +\norm{u_0}_{H^{3}(\Omega_B)}+\norm{u_1}_{H^{3/2}(\Omega_B)}
\right).
\end{split}\end{equation*}
Choose $T\in (0,1)$ such that $1-C_T (R+1)T^a\ge1/2$. We then find that 
\begin{equation}\label{eq:estimate-energy-fluid}
\begin{split}
\norm{v}_{{\mathcal V}_2}+\norm{p}_{\Pi_2}\le & C_1\left[\norm{V_0}_{H^{5/2}(\Omega_L)}+\norm{V_1}_{H^1(\Omega_L)} \right.
\\
&\left.+\norm{Q_0}_{H^{3/2}(\Omega_L)}+\norm{u_0}_{H^{3}(\Omega_B)}+\norm{u_1}_{H^{3/2}(\Omega_B)}\right].
\end{split}\end{equation}
Using the latter in \eqref{eq:estimate-elastic}$_1$ for the estimate of the left-hand side with $\tilde v=v$, we obtain 
\begin{multline}\label{eq:estimate-energy-elastic}
\norm{u}_{\mathcal U}
\le C_2\left[\norm{u_0}_{H^{3}(\Omega_B)}+\norm{u_1}_{H^{3/2}(\Omega_B)}+\norm{V_0}_{H^{5/2}(\Omega_L)}\right.
\\
\left.+\norm{V_1}_{H^1(\Omega_L)}+\norm{Q_0}_{H^{3/2}(\Omega_L)}\right].
\end{multline}
Using \eqref{eq:estimate-compatibility-V1} and \eqref{eq:estimate-compatibility-Q0}, in \eqref{eq:estimate-energy-fluid} and \eqref{eq:estimate-energy-elastic}, we obtain \eqref{eq:energy_hateq}.

To conclude the proof of this theorem, consider the map  
\[
\chi(\cdot,t):\;X\in \Omega_L\mapsto \chi(X,t):=X+\int^t_0v(X,s)\; ds\ \in \Omega_L.
\]
We first notice that $\chi$ is injective, and so $J:=\det(\nabla \chi)\ne 0$. From the regularity of $v$ and the estimate \eqref{eq:energy_hateq}, 
it immediately follows that $\chi\in B_R$ provided that $T$ is small enough and $R\ge \norm{V_0}_{H^{5/2}(\Omega_L)}$ is large enough. 
\end{proof}

\section{Proof of Theorem \ref{th:main} and concluding remarks}\label{sec:proof-main}

In this final section, we will prove Theorem \ref{th:main}. We start by noticing that the strategy employed to prove Theorem \ref{th:existence-eom-L-c-hat} can be adapted to prove the following proposition. 

\begin{proposition}\label{prop:nStokes-hatchi-fgd}
Let $\hat\chi\in B_R$ as in the Section \ref{sec:proof-thm-eom-L-c-hat} (with corresponding $\hat F=\nabla \hat \chi$). For every 
\[\begin{split}
&g\in L^2(0,T;H^2(\Omega_L))\cap H^1(0,T;H^1(\Omega_L))\cap H^2(0,T;H^{-1}(\Omega_L)),
\\
&f\in L^2(0,T;H^1(\Omega_L))\cap H^1(0,T;L^2(\Omega_L)),
\\
&d\in L^2(0,T;H^{3/2}(\Gamma_L))\cap H^1(0,T;H^{1/2}(\Gamma_L))\cap H^{1+1/4}(0,T;L^2(\Gamma_L)),
\\
&w_0\in H^{5/2}(\Omega_L),
\end{split}\] 
there exists a non-decreasing positive function $G=G(t)$, $G(0)=0$, $G(t)\to +\infty$ as $t\to +\infty$ such that for all $T\in (0,1)$ satisfying  $G(T)(R+1)<1$ there exists a unique solution $(w,\pi)\in \mathcal V_2\times \Pi_2$ to the following initial-boundary value problem:
\begin{equation}\label{eq:nStokes-hatchi-fgd}
\begin{aligned}
&\nabla w:\cof{\hat F}=g\qquad &&\text{in }\Omega_L\times (0,T),
\\
&\rho_{_L}\frac{\partial w}{\partial t}=\div \mathcal T_{\hat \chi}(w,\pi)+f\qquad &&\text{in }\Omega_L\times (0,T),
\\
&\mathcal T_{\hat \chi}(w,\pi)\cdot\n=d\qquad &&\text{on }\Gamma_L\times(0,T),
\\
&w(\cdot,0)=w_0\qquad &&\text{in }\Omega_L. 
\end{aligned}
\end{equation}
In the above equation, 
\begin{equation}\label{eq:stress-hat-chi}
\mathcal T_{\hat \chi}(w,\pi):=-\pi\cof{\hat F}+\mu\nabla w\cdot(\cof{\hat F})^T\cdot\cof{\hat F}+\mu\cof{\hat F}\cdot(\nabla w)^T\cdot\cof{\hat F},
\end{equation}
and the data satisfy the following compatibility conditions:
\begin{itemize}
\item There exist $(w_1,\pi_0)\in H^1(\Omega_L)\times H^{3/2}(\Omega_L)$ 
such that 
\begin{equation}\label{eq:compatibility-nStokes-hatchi-fgd}
\begin{aligned}
&\div w_0=g &&\text{in }\Omega_L,
\\
&\div w_1=\nabla w_0 :(\nabla v_0)^T+\frac{\partial g}{\partial t}(\cdot,0)&&\text{in }\Omega_L,
\\
&\rho_{_L} w_1=-\nabla \pi_0+\mu \Delta w_0+f(\cdot,0)&&\text{in }\Omega_L,
\\
&\pi_0=\left[2\mu D(w_0)\cdot\n-d(\cdot,0)\right]\cdot\n&&\text{on }\Gamma_L.
\end{aligned}
\end{equation}
\end{itemize}
Furthermore, the following estimate holds with a positive constant $C$ independent of $T$
\begin{equation*} 
\begin{split}
\norm{w}_{\mathcal V_2}+\norm{\pi}_{\Pi_2}\le C&\left[(1+\norm{v_0}_{H^{5/2}(\Omega_L)})\norm{w_0}_{H^{5/2}(\Omega_L)}+\norm{g}_{L^2(0,T;H^2(\Omega_L))}\right.
\\
&\qquad
+\norm{\frac{\partial g}{\partial t}}_{L^2(0,T;H^{1}(\Omega_L))}+\norm{\frac{\partial^2 g}{\partial t^2}}_{L^2(0,T;H^{-1}(\Omega_L))}
\\
&\qquad +\norm{f}_{L^2(0,T;H^1(\Omega_L))}+\norm{\frac{\partial f}{\partial t}}_{L^2((0,T)\times \Omega_L)}
\\
&\qquad +\norm{d}_{ L^2(0,T;H^{3/2}(\Gamma_L))}+\norm{\frac{\partial d}{\partial t}}_{ L^2(0,T;H^{1/2}(\Gamma_L))}
\\
&\qquad \left.+\norm{d}_{H^{1/4}(0,T;L^2(\Gamma_L))}+\norm{\frac{\partial d}{\partial t}}_{H^{1/4}(0,T;L^2(\Gamma_L))}\right].
\end{split}
\end{equation*}
\end{proposition}

\begin{proof}
As already mentioned, the proof follows closely the strategy of the proof of Theorem  \ref{th:existence-eom-L-c-hat}. We will then just sketch the main steps. For every fixed $(\tilde w,\tilde \pi)\in \mathcal E$ with $V_0=w_0$, $V_1=w_1$ and $Q_0=\pi_0$ (cf. \eqref{eq:B_M}), we solve the non-homogeneous Stokes problem 
\begin{equation*}
\begin{aligned}
&\div w=\tilde g\qquad &&\text{in }\Omega_L\times (0,T),
\\
&\rho_{_L}\frac{\partial w}{\partial t}=\div \mathcal T(w,\pi)+\tilde f\qquad &&\text{in }\Omega_L\times (0,T),
\\
&\mathcal T(w,\pi)\cdot\n=\tilde d\qquad &&\text{on }\Gamma_L\times(0,T),
\\
&w(\cdot,0)=w_0\qquad &&\text{in }\Omega_L,
\end{aligned}
\end{equation*}
where 
\[\begin{split}
&\tilde g:=g+(I-\cof{\hat F}):\nabla \tilde w,
\\
&\tilde f=\div \mathcal T_{\hat \chi}(\tilde w,\tilde \pi)-\div \mathcal T(\tilde w,\tilde \pi)+f,
\\
&\tilde d=-\mathcal T_{\hat \chi}(\tilde w,\tilde \pi)\cdot\n+\mathcal T(\tilde w,\tilde \pi)\cdot\n+d.
\end{split}\]
For the above non-homogeneous Stokes problem, we can then use Theorem \ref{th:nStokes_r} because the forces $\tilde g$, $\tilde f$ and $\tilde d$ are in the ``correct'' functional spaces. The proof of the latter is the same as that of Lemma \ref{lem:estimatefgd}. The fixed point argument and the corresponding energy estimate then follow as in the proof of Theorem  \ref{th:existence-eom-L-c-hat} (we note that the observations made in Remark \ref{rm:fixedpoint0data} and Remark \ref{rm:compatibility-conditions-hatchi} can be easily adapted to the problem at hand). 
\end{proof}

We are now ready to prove the main theorem of this paper. 

\begin{proof}[Proof of Theorem \ref{th:main}]

Consider the map $\Lambda$ defined in \eqref{eq:fixed-point-map-chi}, where $v$ --together with $u$ and $p$-- solves \eqref{eq:eom-L-c-hat} with initial condition $(u_0,u_1,V_0=v_0)$ according to Theorem \ref{th:existence-eom-L-c-hat}. Thanks to \eqref{eq:energy_hateq}, we can then choose $R$ in \eqref{eq:B_R} sufficiently large, and possibly take a smaller $T$ satisfying \eqref{eq:small-time-R}, so that $\Lambda(\hat \chi)=\chi \in B_R$, where
\[
\chi(X,t)=X+\int^t_0v(X,s)\; d s\qquad\text{for all }(X,t)\in \Omega_L\times [0,T). 
\]
It remains to show that $\Lambda$ is a contraction. Since $B_R$ is a closed subset of $\mathcal D$ (defined in \eqref{eq:spaceD}), it is enough to show that there exists a constant $\alpha_T\in (0,1)$ such that 
\begin{equation}\label{eq:contraction-Lambda}
\norm{\Lambda(\hat\chi^{(1)})-\Lambda(\hat\chi^{(2)})}_{\mathcal D}\le \alpha_T\norm{\hat \chi^{(1)}-\hat \chi^{(2)}}_{\mathcal D}
\end{equation}
for all $\hat\chi^{(1)}, \hat\chi^{(2)}\in B_R$ (with corresponding gradients $\hat F^{(1)}$ and $\hat F^{(2)}$), where 
\begin{equation}\label{eq:Lambda-i}
\Lambda(\hat\chi^{(i)})=\chi^{(i)}(X,t)=X+\int^t_0v^{(i)}(X,s)\; d s\quad\text{for all }(X,t)\in \Omega_L\times [0,T),\  i=1,2,
\end{equation}
with $(u^{(i)},v^{(i)},p^{(i)})$ the solution to \eqref{eq:eom-L-c-hat}, corresponding to $\hat\chi^{(i)}$, with initial condition $(u_0,u_1,V_0=v_0)$ (according to Theorem \ref{th:existence-eom-L-c-hat}). 

Set 
\[
\xi:=u^{(1)}-u^{(2)},\qquad w:=v^{(1)}-v^{(2)},\qquad \pi:=p^{(1)}-p^{(2)}.
\]
From \eqref{eq:Lambda-i} and the definition of the norm $\norm{\cdot}_{\mathcal D}$ (see \eqref{eq:normD}$_1$), to prove \eqref{eq:contraction-Lambda}, it is enough to prove that there exists $\beta_T\in (0,1)$ such that 
\[\begin{split}
\norm{w}_{L^2(0,T;H^3(\Omega_L))}&+\norm{\frac{\partial w}{\partial t}}_{L^2(0,T;H^2(\Omega_L))} +\norm{\frac{\partial^2 w}{\partial t^2}}_{L^2(0,T;L^2(\Omega_L))}
\\
\le \beta_T & \left[\norm{\hat v^{(1)}-\hat v^{(2)}}_{L^2(0,T;H^3(\Omega_L))}+\norm{\frac{\partial}{\partial t}\hat v^{(1)}-\frac{\partial}{\partial t}\hat v^{(2)}}_{L^2(0,T;H^2(\Omega_L))} \right.
\\
&\qquad\qquad\qquad\qquad\qquad\qquad \left.+\norm{\frac{\partial^2}{\partial t^2}\hat v^{(1)}-\frac{\partial^2}{\partial t^2}\hat v^{(2)}}_{L^2(0,T;L^2(\Omega_L))}\right],
\end{split}\]
where
\[
\hat v^{(i)}=\frac{\partial \hat\chi^{(i)}}{\partial t}\qquad\text{ for }i=1,2.
\]
We notice that the triple $(\xi,w,\pi)$ satisfies  
\begin{equation}\label{eq:difference-elastic}
\begin{aligned}
&\rho_B\frac{\partial^2 \xi}{\partial t^2}=\div P(\xi) &&\text{in }\Omega_B\times(0,T),
\\
&P(\xi)\cdot \n=0 &&\text{on }\Gamma_B\times(0,T),
\\
&\frac{\partial \xi}{\partial t}=w&&\text{on }\Gamma_L\times(0,T),
\\
&u(\cdot,0)=0,\quad\frac{\partial u}{\partial t}u(\cdot,0)=0&&\text{in }\Omega_B,
\end{aligned}\end{equation}
and 
\begin{equation}\label{eq:difference-flui}
\begin{aligned}
&\nabla w:\cof{\hat F^{(1)}}=\nabla v^{(2)}:[\cof{\hat F^{(2)}}-\cof{\hat F^{(1)}}]&&\text{in }\Omega_L\times (0,T),
\\
&\rho_{_{L}}\frac{\partial w}{\partial t}=\div \mathcal T_{\hat \chi^{(1)}}(w,\pi)+\div \mathcal T_{\hat\chi^{(1)}}(v^{(2)},p^{(2)})-\div \mathcal T_{\hat\chi^{(2)}}(v^{(2)},p^{(2)})&&\text{in }\Omega_L\times (0,T),
\\
&\mathcal T_{\hat\chi^{(1)}}(w,\pi)\cdot \n=[P(\xi)-\mathcal T_{\hat\chi^{(1)}}(v^{(2)},p^{(2)})+\mathcal T_{\hat\chi^{(2)}}(v^{(2)},p^{(2)})]\cdot \n
&&\text{on }\Gamma_L\times(0,T),
\\
&w(\cdot,0)=0&&\text{in }\Omega_L,
\end{aligned}
\end{equation}
where $P$ is defined in \eqref{eq:piola}, and $\mathcal T_{\hat\chi}$ is given in \eqref{eq:stress-hat-chi}. By Proposition \ref{prop:nStokes-hatchi-fgd} with $T\in (0,1)$ satisfying \eqref{eq:small-time-R} (and zero initial conditions), we find that $(w,\pi)$ satisfies the following estimate
\begin{equation}\label{eq:diff-energy}
\begin{split}
\norm{w}_{\mathcal V_2}+&\norm{\pi}_{\Pi_2}\le C\left[\norm{\nabla v^{(2)}:[\cof{\hat F^{(2)}}-\cof{\hat F^{(1)}}]}_{L^2(0,T;H^2(\Omega_L))}\right.
\\
&\quad
+\norm{\frac{\partial }{\partial t}\left\{\nabla v^{(2)}:[\cof{\hat F^{(2)}}-\cof{\hat F^{(1)}}]\right\}}_{L^2(0,T;H^{1}(\Omega_L))}
\\
&\quad +\norm{\frac{\partial^2 }{\partial t^2}\left\{\nabla v^{(2)}:[\cof{\hat F^{(2)}}-\cof{\hat F^{(1)}}]\right\}}_{L^2(0,T;H^{-1}(\Omega_L))}
\\
&\quad +\norm{\div \mathcal T_{\hat\chi^{(1)}}(v^{(2)},p^{(2)})-\div \mathcal T_{\hat\chi^{(2)}}(v^{(2)},p^{(2)})}_{L^2(0,T;H^1(\Omega_L))}
\\
&\quad+\norm{\frac{\partial }{\partial t}\left[\div \mathcal T_{\hat\chi^{(1)}}(v^{(2)},p^{(2)})-\div \mathcal T_{\hat\chi^{(2)}}(v^{(2)},p^{(2)})\right]}_{L^2((0,T)\times \Omega_L)}
\\
&\quad +\norm{P(\xi)\cdot\n}_{ L^2(0,T;H^{3/2}(\Gamma_L))} +\norm{\frac{\partial }{\partial t}P(\xi)\cdot\n}_{ L^2(0,T;H^{1/2}(\Gamma_L))}
\\
&\quad+\norm{[\mathcal T_{\hat\chi^{(1)}}(v^{(2)},p^{(2)})-\mathcal T_{\hat\chi^{(2)}}(v^{(2)},p^{(2)})]\cdot\n}_{ L^2(0,T;H^{3/2}(\Gamma_L))}
\\
&\quad+\norm{\frac{\partial }{\partial t}[\mathcal T_{\hat\chi^{(1)}}(v^{(2)},p^{(2)})-\mathcal T_{\hat\chi^{(2)}}(v^{(2)},p^{(2)})]\cdot\n}_{ L^2(0,T;H^{1/2}(\Gamma_L))}
\\
&\quad+\norm{P(\xi)\cdot\n}_{H^{1/4}(0,T;L^2(\Gamma_L))}+\norm{\frac{\partial}{\partial t}P(\xi)\cdot\n}_{H^{1/4}(0,T;L^2(\Gamma_L))} 
\\
&\quad+\norm{[\mathcal T_{\hat\chi^{(1)}}(v^{(2)},p^{(2)})-\mathcal T_{\hat\chi^{(2)}}(v^{(2)},p^{(2)})]\cdot\n}_{H^{1/4}(0,T;L^2(\Gamma_L))}
\\
&\quad \left.+\norm{\frac{\partial}{\partial t}\left[\mathcal T_{\hat\chi^{(1)}}(v^{(2)},p^{(2)})-\mathcal T_{\hat\chi^{(2)}}(v^{(2)},p^{(2)})\right]\cdot\n}_{H^{1/4}(0,T;L^2(\Gamma_L))}\right].
\end{split}
\end{equation}
From \eqref{eq:Lambda-i} and Jensen's inequality, it follows that there exists a positive constant $c$ independent of $T$ such that 
\begin{equation}\label{eq:diff-hatF-hatFt}\begin{split}
&\norm{\hat F^{(1)}-\hat F^{(2)}}_{C([0,T];H^2(\Omega_L))}\le cT^{1/2}\norm{\hat v^{(1)}-\hat v^{(2)}}_{L^2(0,T;H^3(\Omega_L))},
\\
&\norm{\frac{\partial}{\partial t}\hat F^{(1)}-\frac{\partial}{\partial t}\hat F^{(2)}}_{C([0,T];H^1(\Omega_L))}\le cT^{1/2}\norm{\frac{\partial}{\partial t}\hat v^{(1)}-\frac{\partial}{\partial t}\hat v^{(2)}}_{L^2(0,T;H^1(\Omega_L))}.
\end{split}\end{equation}
In addition, 
we have
\[\frac{\partial}{\partial t}\hat v^{(1)}-\frac{\partial}{\partial t} \hat v^{(2)}\in L^2(0,T;H^2(\Omega_L))\cap H^1(0,T;L^2(\Omega_L))
\hookrightarrow C([0,T];H^1(\Omega_L))
\]
since  
\[\hat v^{(1)}(\cdot, 0)-\hat v^{(2)}(\cdot, 0)=0,\qquad \frac{\partial}{\partial t}\hat v^{(1)}(\cdot, 0)- \frac{\partial}{\partial t}\hat v^{(2)}(\cdot, 0)=0
\quad\text{ in }\Omega_L. 
\]
Then, by Lemma \ref{lem:holder-st} (with $p=2$, $s_1=1$, $q=+\infty$ and $s_2=0$), we obtain that 
\begin{equation}\label{eq:diff-hatFtt}\begin{split}
&\norm{\frac{\partial^2}{\partial t^2}\hat F^{(1)}-\frac{\partial^2}{\partial t^2}\hat F^{(2)}}_{L^{2/\theta}(0,T;H^{\theta}(\Omega_L))}
\\
&\quad \le c_1\norm{\frac{\partial}{\partial t}\hat v^{(1)}-\frac{\partial}{\partial t}\hat v^{(2)}}^\theta_{L^2(0,T;H^1(\Omega_L))}\norm{\frac{\partial }{\partial t}\hat v^{(1)}-\frac{\partial}{\partial t}\hat v^{(2)}}^{1-\theta}_{C([0,T];L^2(\Omega_L))}, 
\end{split}\end{equation}
for $\theta\in[1/2,1]$. In particular, choosing $\theta=3/4$, and using and Sobolev embedding theorem and Proposition \ref{prop:traceat0}, we have that 
\begin{equation}\label{eq:diff-hatF-tt2}
\begin{split}
&\norm{\frac{\partial^2}{\partial t^2}\hat F^{(1)}-\frac{\partial^2}{\partial t^2}\hat F^{(2)}}_{L^{2}(0,T;L^3(\Omega_L))}
\le c_2 T^{1/8}\left[\norm{\frac{\partial}{\partial t}\hat v^{(1)}-\frac{\partial}{\partial t}\hat v^{(2)}}_{L^2(0,T;H^2(\Omega_L))}
\right.
\\
&\qquad\qquad\qquad\qquad\qquad\qquad\qquad\qquad\qquad\left.+\norm{\frac{\partial^2}{\partial t^2}\hat v^{(1)}-\frac{\partial^2}{\partial t^2}\hat v^{(2)}}_{L^2(0,T;L^2(\Omega_L))}\right]. 
\end{split}\end{equation}

We further notice that, from Lemma \ref{lem:estimate_hatchi} part $(ii)$ and \eqref{eq:iv1}, 
\[
\norm{\frac{\partial}{\partial t}\nabla \hat F^{(i)}}_{L^2(0,T;H^2(\Omega_L))}+
\norm{\hat F^{(i)}}_{C([0,T];H^2(\Omega_L))}\le C(R+1)\qquad\text{for }i=1,2.
\]
Using the latter with \eqref{eq:diff-hatF-hatFt} and \eqref{eq:diff-hatFtt} (with $\theta=1$), we obtain 
\begin{multline}\label{eq:diff-cof-ttx}
\norm{\cof{\hat F^{(1)}}-\cof{\hat F^{(2)}}}_{H^2(0,T;H^1(\Omega_L))}
\\
\le 
c_3(R+1)\left[\norm{\hat v^{(1)}-\hat v^{(2)}}_{L^2(0,T;H^3(\Omega_L))} + \norm{\frac{\partial}{\partial t}\hat v^{(1)}-\frac{\partial}{\partial t}\hat v^{(2)}}_{L^2(0,T;H^1(\Omega_L))}\right]. 
\end{multline}

We then apply Theorem \ref{th:N} to \eqref{eq:difference-elastic}. By Remark \ref{rm:trace-regularity-v} and estimate \eqref{eq:estimate-N}, trace theorem, the fact that $T\in (0,1)$ and that $w(\cdot,0)=0$, and by Lemma \ref{lem:interpolation}, we have that $\xi$ in \eqref{eq:difference-elastic} satisfies the following estimates with a constant $c_i$, $i=4,\dots,8$,  independent of $T$: 
\begin{equation}\label{eq:diff-P-tx/2}\begin{split}
&\norm{P(\xi)\cdot n}_{H^1(0,T;H^{1/2}(\Gamma_L))} 
\\
& \quad \le c_4T^{1/2}\left[T\norm{w}_{L^2(0,T;H^{3}(\Omega_L))}+\norm{w}_{H^{3/2}(0,T;H^1(\Omega_L))}\right]+ \norm{w}_{L^2(0,T;H^{2}(\Omega_L))}
\\
&\quad \le c_5T^{1/2}\left[\norm{w}_{L^2(0,T;H^{3}(\Omega_L)})+\norm{w}_{H^{3/2}(0,T;H^1(\Omega_L))}+ \norm{\frac{\partial w}{\partial t}}_{L^2(0,T;H^{2}(\Omega_L))}\right]
\\
&\quad \le c_6T^{1/2}\left[\norm{w}_{L^2(0,T;H^{3}(\Omega_L))}
+ \norm{\frac{\partial w}{\partial t}}_{L^2(0,T;H^{2}(\Omega_L))}
+\norm{\frac{\partial^2w}{\partial t^2}}_{L^2(0,T;L^2(\Omega_L))}\right].
\end{split}\end{equation}
In addition, by \eqref{eq:estimate-N}$_2$ and Lemma \ref{lem:tsigma-s} (with $\sigma=1$ and $s=1/2$, since $w(\cdot,0)=0$ in $\Omega_L$), we also have that 
\begin{equation}\label{eq:diff-P-1.5}\begin{split}
\norm{P(\xi)\cdot\n}_{H^{3/2}(0,T;L^2(\Gamma_L))}&+\norm{P(\xi)\cdot\n}_{L^2(0,T;H^{3/2}(\Gamma_L))}
\\
&\quad \le c_7\left[T^{1/2}\norm{w}_{L^2(0,T;H^{3}(\Omega_L))}+\norm{w}_{H^{1/2}(0,T;H^2(\Omega_L))}\right]
\\
&\quad \le c_8T^{1/2}\left[\norm{w}_{L^2(0,T;H^{3}(\Omega_L))}+\norm{\frac{\partial w}{\partial t}}_{L^2(0,T;H^2(\Omega_L)}\right].
\end{split}\end{equation}

We finally note that, by Remark \ref{rm:trace-regularity-v} and \eqref{eq:energy_hateq}, it follows that 
\begin{multline*}
\norm{v^{(2)}}_{H^{1/4}(0,T;H^1(\Gamma_L))}+\norm{\frac{\partial v^{(2)}}{\partial t}}_{H^{1/4}(0,T;H^1(\Gamma_L))}
\\
+\norm{p}_{H^{1/4}(0,T;H^1(\Gamma_L))}+\norm{\frac{\partial p}{\partial t}}_{H^{1/4}(0,T;L^2(\Gamma_L))}\le cR.
\end{multline*}
Using the latter, and the estimates \eqref{eq:diff-hatF-hatFt}, \eqref{eq:diff-hatF-tt2}, \eqref{eq:diff-cof-ttx}, \eqref{eq:diff-P-tx/2}, and  \eqref{eq:diff-P-1.5} together with \eqref{eq:energy_hateq} on the right-hand side \eqref{eq:diff-energy}, we can find a constant $c>0$ (independent of $T$) and $a\in (0,1)$ such that 
\[\begin{split}
\norm{w}_{\mathcal V_2}&+\norm{\pi}_{\Pi_2}\le cT^a \left[\norm{\hat v^{(1)}-\hat v^{(2)}}_{L^2(0,T;H^3(\Omega_L))}\right. 
\\
&\left. +\norm{\frac{\partial}{\partial t}\hat v^{(1)}-\frac{\partial}{\partial t}\hat v^{(2)}}_{L^2(0,T;H^2(\Omega_L))}
+\norm{\frac{\partial^2}{\partial t^2}\hat v^{(1)}-\frac{\partial^2}{\partial t^2}\hat v^{(2)}}_{L^2(0,T;L^2(\Omega_L))}\right], 
\end{split}\]
and from this, the contraction property follows by choosing $T$ sufficiently small.  
\end{proof}

\begin{remark}\label{rm:forces}
The strategy used in this paper can be used to prove the existence of a local (in time) solution to  \eqref{eq:eom-L-c}--\eqref{eq:ic-L} when external body forces are applied to the fluid and/or on the solid. This methodology would allow the proof of existence of solutions to the corresponding forced system\footnote{Such a system is obtained by adding a forcing term $F_1$ and $F_2$ at the right-hand side of \eqref{eq:eom-L-c}$_1$ and \eqref{eq:eom-L-c}$_3$, respectively. }, in the same regularity class as Theorem \ref{th:main}, if the fluid and the solid are subject to body forces $F_1$ and $F_2$ with regularity
\[\begin{split}
&F_1\in L^2(0,T;H^{3/2}(\Omega_B))\cap H^{3/2}(0,T;H^{1/2}(\Omega_B)),
\\
&F_2 \in  L^2(0,T;H^1(\Omega_L))\cap H^1(0,T;L^2(\Omega_L)),
\end{split}\]
respectively (and satisfying suitable compatibility conditions). We do not present the details here because this does not seem to add any further difficulty to the mathematical treatment. 
\end{remark}

We conclude this paper by mentioning the following result on continuous dependence upon initial data. We first notice that uniqueness of solutions to \eqref{eq:eom-L-c}--\eqref{eq:ic-L} is an immediate consequence of the uniqueness of fixed points guaranteed by Banach fixed point theorem. Although they are linear in the variables $(u,v,p)$, equations \eqref{eq:eom-L-c}--\eqref{eq:ic-L} are non-linear with respect to the variable $\chi$. Nevertheless, using an argument similar to the contraction argument in the latter proof and invoking Proposition \ref{prop:nStokes-hatchi-fgd}, the following corollary can be easily obtained.

\begin{corollary}\label{cor:cont-dep-data}
Let $(u_0^{(1)},u_1^{(1)},v_0^{(1)})$ and $(u_0^{(2)},u_1^{(2)},v_0^{(2)})\in H^{3}(\Omega_B)\times H^{3/2}(\Omega_B) \times H^{5/2}(\Omega_L)$ satisfying the {\em compatibility conditions} (i)--(iv) in Theorem \ref{th:main}, and let $(u^{(1)},v^{(1)},p^{(1)},\chi^{(1)})$ and $(u^{(2)},v^{(2)},p^{(2)},\chi^{(2)})$ be the corresponding solutions (according to Theorem \ref{th:main}) defined on the interval $[0,T^{(1)})$ and $[0,T^{(2)})$, respectively. Set $T:=\min\{T^{(1)},T^{(2)}\}$. Then, there exists a constant $C$, independent of $T$ such that 
\begin{equation}\label{eq:cont-dep-data}
\begin{split}
&\norm{u^{(1)}-u^{(2)}}_{\mathcal U}+\norm{v^{(1)}-v^{(2)}}_{\mathcal V_2}+\norm{p^{(1)}-p^{(2)}}_{\Pi_2}
\\
&\qquad \le C\left[(1+\norm{v^{(1)}_0-v^{(2)}_0}_{H^{5/2}(\Omega_L)})\norm{v^{(1)}_0-v^{(2)}_0}_{H^{5/2}(\Omega_L)}\right.
\\
&\qquad\qquad \left.+\norm{u^{(1)}_0-u^{(2)}_0}_{H^{3}(\Omega_B)}+\norm{u^{(1)}_1-u^{(2)}_1}_{H^{3/2}(\Omega_B)}\right].
\end{split}
\end{equation}
\end{corollary}

\begin{appendices}
\section{On the non-homogeneous Stokes problem and the Navier equations}\label{ap:stokes-navier}
In this appendix, we revisit some classical results on the non-homogeneous Stokes problem and some regularity results on the Navier equations of elasticity that play an important role in our paper. 

Assume that $D$ is a bounded domain in $\R^3$. In our proofs, we need the following theorems concerning the {\em non-homogeneous Stokes problem}: given $(g,f,h,v_0)$ find $(v,p)$ satisfying the following initial-boundary value problem
\begin{equation}\label{eq:n-S}
\begin{aligned}
&\div v=g &&\text{in }D\times(0,T)
\\
&\rho_{_L}\frac{\partial v}{\partial t}-\div \mathcal T(v,p)
=f&&\text{in }D\times(0,T),
\\
&\mathcal T(v,p)\cdot \n=d&&\text{on }\partial D\times(0,T),
\\
&v(\cdot,0)=v_0&&\text{in }D.
\end{aligned}
\end{equation}
In the above equation, $\rho_{_L}$ is positive constant, $\mathcal T$ is the Cauchy stress tensor for viscous incompressible fluids defined in \eqref{eq:cauchy}, and $\n$ denotes the unit outward normal to $\partial D$.

\begin{theorem}\label{th:nStokes}
Let $D$ be a smooth\footnote{The regularity of the boundary could be further lowered, see \cite[Theorem 2 \& comments in Section 5]{solo88}. } domain in $\R^3$ and $T>0$. Consider 
\begin{equation}\label{eq:data_n-S}
\begin{split}
&g\in L^2(0,T;H^1(D))\cap H^1(0,T;H^{-1}(D)),
\\
&f\in L^2(D\times(0,T)),
\\
&d\in L^2(0,T;H^{1/2}(\partial D))\cap H^{1/4}(0,T;L^2(\partial D)),
\\
&v_0\in H^1(D),
\end{split}
\end{equation}
satisfying the {\em compatibility condition}
\[
\div v_0=g(\cdot,0)\qquad \text{in }D.
\]
Then, there exists a unique solution $(v,p)$ to \eqref{eq:n-S} with 
\[\begin{split}
&v\in \mathcal V_1:= L^2(0,T;H^2(D))\cap H^1(0,T;L^2(D)) \cap C([0,T];H^1(D)),
\\
&p\in \Pi_1:= L^2(0,T;H^1(D))\cap H^{1/4}(0,T;L^2(\partial D)). 
\end{split}\]
In addition, $(v,p)$ satisfies the following estimate
\begin{equation}\label{eq:estimate_n-S}
\begin{split}
\norm{v}_{{\mathcal V}_1}+\norm{p}_{\Pi_1}\le C_T&\left(\norm{v_0}_{H^1(D)}+\norm{f}_{L^2(D\times(0,T))}\right.
\\
&\quad +\norm{g}_{L^2(0,T;H^1(D))}+\norm{\frac{\partial g}{\partial t}}_{L^2(0,T;H^{-1}(D))}
\\
&\quad \left.+\norm{d}_{ L^2(0,T;H^{1/2}(\partial D))}+\norm{d}_{H^{1/4}(0,T;L^2(\partial D))}\right),
\end{split}\end{equation}
where $C_T$ is a positive constant depending on $T$ (and non-decreasing with $T$), and the norms on the left-hand side are defined as follows 
\begin{equation}\label{eq:normV1}
\norm{v}_{{\mathcal V}_1}:=\norm{v}_{L^2(0,T;H^2(D))}+\norm{\frac{\partial v}{\partial t}}_{L^2(0,T;L^2(D))}+\norm{v}_{C([0,T];H^1(D))},
\end{equation}
\begin{equation}\label{eq:normPi1}
\norm{p}_{\Pi_1}:=\norm{p}_{L^2(0,T;H^1(D))}+\norm{p}_{H^{1/4}(0,T;L^2(\partial D))}.
\end{equation}
\end{theorem}

\begin{proof}
A proof of the corresponding result in the $L^q$-setting, with summability exponent $q>3$ (in our case $q=2$), can be traced back to the works \cite[Sections 2 \& 3]{solo78}, \cite[Theorem 2 \& remarks at the end of page 402]{solo88}, \cite[Theorem 1]{mucha00}). The same result for the problem where $\partial D=\Gamma_c\cup \Gamma_f$, with $\Gamma_c\cap\Gamma_f=\emptyset$, condition \eqref{eq:n-S}$_3$ on $\Gamma_c$ and $v=0$ on $\Gamma_f$, has been extensively studied in \cite[Chapter 1]{kukavica18}. As this is one of the crucial results used in the current paper, we would like to provide its proof here. 

The unique solution to \eqref{eq:n-S} can be written in the form $v:=(w+\nabla z,p)$ where the scalar field $z$ solves the Poisson equation 
\begin{equation}\label{eq:poisson}
\begin{aligned}
\Delta z&=g &&\text{in }D\times(0,T),
\\
z&=0 &&\text{on }\partial D\times(0,T);
\end{aligned}
\end{equation}
whereas $(v,p)$ satisfies the Stokes equation
\begin{equation}\label{eq:stokes}
\begin{aligned}
&\div w=0&&\text{in }D\times(0,T),
\\
&\rho_{_L}\frac{\partial w}{\partial t}-\div \mathcal T(w,p)
=f-\rho_{_L}\frac{\partial\nabla z}{\partial t}+2\mu\div(D(\nabla z))=:\tilde f&&\text{in }D\times(0,T),
\\
&\mathcal T(w,p)\cdot \n=d-2\mu D(\nabla z)\cdot \n=:\tilde d &&\text{on }\partial D\times(0,T),
\\
&w(x,0)=v_0(x)-\nabla z(x,0) =:w_0(x) &&\text{on }D.
\end{aligned}
\end{equation}
Standard elliptic estimates applied to \eqref{eq:poisson}, with $g$ satisfying \eqref{eq:data_n-S}, imply the existence of a solution $z\in L^2(0,T;H^3(D))\cap H^1(0,T;H^1(D))$ satisfying \eqref{eq:poisson}$_1$ a.e. in $D\times(0,T)$ and the boundary condition in the trace sense. In addition, the following estimate holds
\begin{equation}\label{eq:estimate_poisson}
\begin{split}
\norm{z}_{L^2(0,T;H^3(D))}+&\norm{\frac{\partial z}{\partial t}}_{L^2(0,T;H^1(D))} 
\\
&\le k_1\left(\norm{g}_{L^2(0,T;H^1(D))}+\norm{\frac{\partial g}{\partial t}}_{L^2(0,T;H^{-1}(D))}\right),
\end{split}
\end{equation} 
with a positive constant $k_1$ independent of $T$. We further note that, by the compatibility condition, $g(\cdot,0)=\div v_0\in L^2(D)$. Hence, \eqref{eq:poisson} is satisfied in particular at time $t=0$. Invoking again standard elliptic regularity, we have that 
\[
\norm{z(\cdot,0)}_{H^2(D)}\le k_2\norm{v_0}_{H^1(D)},
\]
where $k_2$ is a positive constant. As a consequence, $w_0=v_0-\nabla z(\cdot,0)\in H^1(D)$. Moreover, by Lemma \ref{lem:interpolation} part $(I1)$, we have that $z\in H^\theta(0,T;H^m(D))$ for all $\theta\in [0,1]$ and $m=3-2\theta$. This fact with the (space-)trace theorem implies that
\[
\nabla(\nabla z)\in L^2(0,T;H^{1/2}(\partial D))\cap H^{1/4}(0,T;L^2(\partial D)),
\]
and 
\begin{equation}\label{eq:nablazCH1}
\begin{split}
\norm{\nabla(\nabla z)}_{L^2(0,T;H^{1/2}(\partial D))}&\le k_3\norm{z}_{L^2(0,T;H^{3}(D))},
\\
\norm{\nabla(\nabla z)}_{H^{1/4}(0,T;L^2(\partial D))}&\le k_4\left(
\norm{z}_{L^2(0,T;H^3(D))}+\norm{\frac{\partial z}{\partial t}}_{L^2(0,T;H^1(D))}\right),
\end{split}\end{equation}
for some positive constants $k_3$ and $k_4$, both independent of $T$. 

From the above considerations, we obtain that 
\[\begin{split}
&\tilde f\in L^2(D\times(0,T)),
\\
&\tilde d\in L^2(0,T;H^{1/2}(\partial D))\cap H^{1/4}(0,T;L^2(\partial D)),
\end{split}\]
with 
\begin{equation}\label{eq:estimate_tildeftilded}\begin{split}
&\norm{\tilde f}_{L^2(D\times(0,T))}\le \norm{f}_{L^2(D\times(0,T))}+k_5\left(\norm{g}_{L^2(0,T;H^1(D))}
+\norm{\frac{\partial g}{\partial t}}_{L^2(0,T;H^{-1}(D))}\right),
\\
&\norm{\tilde d}_{L^2(0,T;H^{1/2}(\partial D))}+\norm{\tilde d}_{H^{1/4}(0,T;L^2(\partial D))}
\\
&\qquad \qquad \le \norm{d}_{L^2(0,T;H^{1/2}(\partial D))}+\norm{d}_{H^{1/4}(0,T;L^2(\partial D))}
\\
&\quad \qquad\qquad \qquad+2\mu\left(\norm{\nabla(\nabla z)}_{L^2(0,T;H^{1/2}(\partial D))}+\norm{\nabla(\nabla z)}_{H^{1/4}(0,T;L^2(\partial D))}\right)
\\
&\qquad \qquad\le \norm{d}_{L^2(0,T;H^{1/2}(\partial D))}+\norm{d}_{H^{1/4}(0,T;L^2(\partial D))}
\\
&\quad \qquad\qquad \qquad\qquad \qquad \quad + k_6\left(\norm{g}_{L^2(0,T;H^1(D))}+\norm{\frac{\partial g}{\partial t}}_{L^2(0,T;H^{-1}(D))}\right). 
\end{split}\end{equation}

For the solvability of the Stokes problem \eqref{eq:stokes}, we then use \cite[Theorem 7.5 (with $r=0$)]{grubb_solo}, to conclude that there exists a unique solution $(w,p)$ satisfying the estimates
\begin{equation}\label{eq:estimate_stokes}
\begin{split}
\norm{w}_{{\mathcal V}_1}+\norm{p}_{\Pi_1}\le k_T&\left(\norm{v_0}_{H^1(D)}+\norm{\tilde f}_{L^2(D\times(0,T))}\right.
\\
&\quad \left.+\norm{\tilde d}_{ L^2(0,T;H^{1/2}(\partial D))}+\norm{\tilde d}_{H^{1/4}(0,T;L^2(\partial D))}\right)
\end{split}\end{equation}
for some positive constant $k_T$ (depending on $T$, and non-decreasing with $T$). 
The combination of \eqref{eq:estimate_poisson} and \eqref{eq:nablazCH1} together with \eqref{eq:estimate_stokes} and \eqref{eq:estimate_tildeftilded} gives \eqref{eq:estimate_n-S}. 
\end{proof}

If the forcing terms in \eqref{eq:n-S} allow for time differentiation, we can apply the above theorem to the equations obtained from \eqref{eq:n-S} after differentiating each equation with respect to time, thus obtaining more regular solutions to \eqref{eq:n-S} (see \cite[Theorem 1.4.5]{kukavica18}. 
\begin{theorem}\label{th:nStokes_r}
Let $D$ be a smooth domain in $\R^3$, and take $T>0$. Consider 
\begin{equation}\label{eq:data_n-S2}
\begin{split}
&g\in L^2(0,T;H^2(D))\cap H^1(0,T;H^1(D))\cap H^2(0,T;H^{-1}(D)),
\\
&f\in L^2(0,T;H^1(D))\cap H^1(0,T;L^2(D)), 
\\
&d\in L^2(0,T;H^{1+1/2}(\partial D))\cap H^1(0,T;H^{1/2}(\partial D))\cap H^{1+1/4}(0,T;L^2(\partial D)),
\\
&v_0\in H^{5/2}(D),
\end{split}
\end{equation}
satisfying the following {\em compatibility conditions}
\begin{itemize}
\item[(C1)] $\div v_0=g(\cdot,0)$ in $D$. 
\item[(C2)] There exist $v_1\in H^{1}(D)$ and $p_0\in H^{3/2}(D)$ such that
\begin{equation}\label{eq:compatibility_Stokes}
\begin{aligned}
&\div v_1=\frac{\partial g}{\partial t}(\cdot,0)&&\text{in }D,
\\
&\rho_{_L}v_1=\mu\Delta v_0-\nabla p_0+f(\cdot,0) &&\text{in }D,  
\\
&p_0=\left[2\mu D(v_0)\cdot \n-d(\cdot, 0)\right]\cdot\n&&\text{on }\partial D.
\end{aligned}
\end{equation}
\end{itemize}
Then, there exists a unique solution $(v,p)$ to \eqref{eq:n-S} satisfying 
\begin{equation}\label{eq:spaces2}
\begin{split}
&v\in \mathcal V_2:=L^2(0,T;H^3(D))\cap H^1(0,T;H^2(D)) \cap H^2(0,T;L^2(D))
\\
&p\in \Pi_2:=L^2(0,T;H^2(D))\cap H^1(0,T;H^1(D))\cap H^{1+1/4}(0,T;L^2(\partial D)). 
\end{split}\end{equation}
In addition, $(v,p)$ satisfies the following estimate
\begin{equation}\label{eq:estimate_n-S}
\begin{split}
\norm{v}_{{\mathcal V}_2}+\norm{p}_{\Pi_2}\le C_T&\left(\norm{v_0}_{H^1(D)}+\norm{v_1}_{H^1(D)}+\norm{g}_{L^2(0,T;H^2(D))}\right.
\\
&\quad +\norm{\frac{\partial g}{\partial t}}_{L^2(0,T;H^{1}(D))}+\norm{\frac{\partial^2 g}{\partial t^2}}_{L^2(0,T;H^{-1}(D))}
\\
&\quad +\norm{f}_{L^2(0,T;H^1(D))}+\norm{\frac{\partial f}{\partial t}}_{L^2((0,T)\times D)}
\\
&\quad +\norm{d}_{ L^2(0,T;H^{1+1/2}(\partial D))}+\norm{\frac{\partial d}{\partial t}}_{ L^2(0,T;H^{1/2}(\partial D))}
\\
&\quad \left.+\norm{d}_{H^{1/4}(0,T;L^2(\partial D))}+\norm{\frac{\partial d}{\partial t}}_{H^{1/4}(0,T;L^2(\partial D))}\right),
\end{split}\end{equation}
where $C_T$ is a positive constant depending on $T$ (and non-decreasing with $T$), and the norms on the left-hand side are defined in \eqref{eq:normV2} and \eqref{eq:normPi2}, with $\Omega_L$ now replaced by $D$. 
\end{theorem}

\begin{remark}\label{rm:trace-regularity-v}
We observe that under the regularity stated in the above theorem, by Lemma \ref{lem:interpolation} and Sobolev embedding, we get 
\begin{align*}
v\in L^2(0,T;H^3(D))\cap H^1(0,T;H^2(D))&\hookrightarrow H^\theta(0,T;H^{3-\theta}(D)) &&\text{for }\theta\in [0,1],
\\
\frac{\partial v}{\partial t}\in L^2(0,T;H^2(D))\cap H^1(0,T;L^2(D))&\hookrightarrow H^{\theta}(0,T;H^{2-2\theta}(D)) &&\text{for }\theta\in [0,1],
\\
p\in L^2(0,T;H^2(D))\cap H^1(0,T;H^1(D))&\hookrightarrow H^\theta(0,T;H^{2-\theta}(D))&&\text{for }\theta\in [0,1].
\end{align*}

By the (space-)trace theorem, we also have that 
\[
\norm{v|_{\partial D}}_{H^\theta(0,T;H^{5/2-\theta}(\partial D))}+\norm{v|_{\partial D}}_{H^{s+1}(0,T;H^{3/2-2s}(\partial D))}\le C\norm{v}_{\mathcal V_2}
\]
for $\theta\in [0,1]$ and $s\in [0,3/4)$, where the positive constant $C$ is independent of $T$. In addition, 
\[
v|_{\partial D}\in L^2(0,T;H^{5/2}(\partial D))\cap H^{3/2}(0,T;H^{1/2}(\partial D)).
\]
In view of our next discussion, consider the function $h:(x,t)\mapsto\displaystyle \int^t_0v|_{\partial D}(x,s)ds$. Then, $h\in H^{5/2,5/2}(\partial D\times(0,T))$.
\end{remark}

Let $A$ be a smooth bounded domain in $\R^3$ with boundary $\partial A=\Gamma_0\cup \Gamma_1$ satisfying $ \Gamma_0\ne \emptyset$, $\Gamma_1\ne \emptyset$ and $\Gamma_0\cap \Gamma_1\ne \emptyset$. For our proofs, we also need the following result concerning the {\em Navier equations of linearized elasticity}: given $(h,u_0,u_1)$ find a solution $u$ to the following initial-boundary value problem
\begin{equation}\label{eq:N}
\begin{aligned}
&\rho_B\frac{\partial^2 u}{\partial t^2}=\div P(u) &&\text{in }A\times(0,T),
\\
&P(u)\cdot \n=0 &&\text{on }\Gamma_0\times(0,T),
\\
&\frac{\partial u}{\partial t}=w &&\text{on }\Gamma_1\times(0,T),
\\
&u(\cdot,0)=u_0 &&\text{in }A,
\\
&\frac{\partial u}{\partial t}(\cdot,0)=u_1 &&\text{in }A,
\end{aligned}
\end{equation}
where $\rho_B$ is a positive constant, $P$ is the first Piola-Kirchhoff stress and it is defined in \eqref{eq:piola}. The following theorem holds. 
\begin{theorem}\label{th:N}
Suppose that the data  
\[\begin{split}
(u_0,u_1)&\in H^{3}(A)\times H^{3/2}(A)
\\
w&\in  L^2(0,T;H^{5/2}(\Gamma_1))\cap H^{3/2}(0,T;H^{1/2}(\Gamma_1)) 
\end{split}\] 
satisfy the {\em compatibility conditions} 
\begin{equation}\label{eq:compatibility_N}\begin{aligned}
&P(u_0)\cdot \n=0 &&\text{on }\Gamma_0\times(0,T),
\\
&u_1=w &&\text{on }\Gamma_1\times[0,T),
\\
&\frac{\partial w}{\partial t}(\cdot,0)=\frac{1}{\rho_B}\div P(u_0)&&\text{on }\Gamma_1.
\end{aligned}\end{equation}
Then, there exists a unique solution $u$ to \eqref{eq:N} satisfying the following regularity properties 
\[\begin{split}
u&\in C([0,T]; H^{5/2}(A)),
\\
\frac{\partial u}{\partial t}&\in C([0,T]; H^{3/2}(A)),
\\
\frac{\partial^2 u}{\partial t^2}&\in C([0,T]; H^{1/2}(A)),
\\
P(u)\cdot \n&\in H^{3/2,3/2}(\Gamma_1\times(0,T))\cap H^1(0,T;H^{1/2}(\Gamma_1)),
\end{split}\]
and the estimates
\begin{equation}\label{eq:estimate-N}
\begin{split}
&\norm{u}_{\mathcal U}\le C\left[(1+T^{1/2})\norm{u_0}_{H^{3}(A)}+\norm{u_1}_{H^{3/2}(A)}\right.
\\
&\qquad\qquad\qquad\qquad\qquad\quad\left.+T\norm{w}_{L^2(0,T;H^{5/2}(\Gamma_1))}+\norm{w}_{H^{3/2}(0,T;L^2(\Gamma_1))}\right],
\\
&\norm{P(u)\cdot \n}_{H^{3/2,3/2}(\Gamma_1\times(0,T))}\le C\left[T^{1/2}\norm{u_0}_{H^{3}(A)}+T\norm{w}_{L^2(0,T;H^{5/2}(\Gamma_1))}\right.
\\
&\qquad\qquad\qquad\qquad\qquad\qquad\qquad\qquad\qquad\qquad\qquad\quad\left.+\norm{w}_{H^{1/2}(0,T;H^1(\Gamma_1))}\right],
\\
&\norm{P(u)\cdot \n}_{H^1(0,T;H^{1/2}(\Gamma_1))} \le C\left[T^{1/2}\norm{u}_{\mathcal U}+ \norm{w}_{L^2(0,T;H^{3/2}(\Gamma_1))}\right], 
\end{split}
\end{equation}
for some positive constant $C$ independent of $T$. In the above estimate, the norm $\norm{\cdot}_{\mathcal U}$ is defined in \eqref{eq:normU} with $\Omega_B$ and $\Gamma_L$ now replaced by $A$ and $\Gamma_1$, respectively. 
\end{theorem}

\begin{proof}
We first note that the function $h:(x,t)\mapsto\displaystyle \int^t_0w(x,s)ds+u_0(X)$ satisfies $h\in H^{5/2,5/2}(\Gamma_1\times(0,T))$. We then consider the initial-boundary value problem 
\begin{equation}\label{eq:Ni}
\begin{aligned}
&\rho_B\frac{\partial^2 u}{\partial t^2}=\div P(u) &&\text{in }A\times(0,T),
\\
&P(u)\cdot \n=0 &&\text{on }\Gamma_0\times(0,T),
\\
&u=h &&\text{on }\Gamma_1\times(0,T),
\\
&u(\cdot,0)=u_0 &&\text{in }A,
\\
&\frac{\partial u}{\partial t}(\cdot,0)=u_1 &&\text{in }A.
\end{aligned}
\end{equation}
Existence and uniqueness of the solution $u$ to \eqref{eq:Ni} is classical, see e.g. \cite{lionsmagenesI}. Concerning the regularity of $u$, the sharp trace regularity results in \cite{LLT}  follow also for \eqref{eq:Ni}, yielding the following regularity for $u$ (cf. \cite[Remark 2.10]{LLT}, see also \cite[Section 3]{raymond2014}): 
\[\begin{split}
&u\in C([0,T];H^{5/2}(A)),
\\
&\frac{\partial u}{\partial t}\in C([0,T];H^{3/2}(A)),
\\
&\frac{\partial^2u}{\partial t^2}\in C([0,T];H^{1/2}(A)),
\\
&P\left(u\right)\cdot n\in H^{3/2,3/2}(\Gamma_1\times (0,T)),
\end{split}\]
with 
\[
\norm{u}_{\mathcal U}
\le C\left[\norm{u_0}_{H^{5/2}(A)}+\norm{u_1}_{H^{3/2}(A)}+\norm{h}_{H^{5/2,5/2}(\Gamma_1\times(0,T))}\right],
\]
for some positive constant $C$ independent on $T$ (the fact that the constant is independent of $T$ can be deduced as in \cite[Theorem 3.2 and comments immediately before]{raymond2014} plus an interpolation argument). From the definition of $h$, the above embedding, trace theorem, and Jensen's inequality, we have that 
\begin{multline*}
\norm{h}_{H^{5/2,5/2}(\Gamma_1\times(0,T))}\le T\norm{w}_{L^2(0,T;H^{5/2}(\Gamma_1))}+\sqrt 2\sqrt T \norm{u_0}_{H^3(A)}
\\
+\norm{w}_{H^{3/2}(0,T;L^2(\Gamma_1))}.
\end{multline*}
We also note that 
$\partial u/\partial t\in C([0,T];H^{1}(\Gamma_1))$. In addition, by Lemma \ref{lem:interpolation} part (I1) (with $m=1$ and $\theta=3/4$), 
\[
w\in H^{1/2+1/4}(0,T;H^1(\Gamma_1))
\hookrightarrow C([0,T];H^1(\Gamma_1)). 
\]
Thus, 
\[
\frac{\partial u}{\partial t}=w\qquad \text{on }\Gamma_1\times (0,T).
\]
Thanks to the regularity of $P(u)\cdot\n$ on the boundary, the above boundary condition and the regularity of $w$, we easily deduce \eqref{eq:estimate-N}$_2$. In addition,  
we have that  
\[
\frac{\partial}{\partial t}(P(u)\cdot n)\in L^2(0,T;H^{3/2}(\Gamma_1)). 
\]
In particular, $\norm{P(u)\cdot n}_{H^1(0,T;H^{1/2}(\Gamma_1))}\le c\left(T^{1/2}\norm{u}_{\mathcal U}+\norm{w}_{L^2(0,T;H^{3/2}(\Gamma_1))}\right)$, for some positive constant $c$, independent of $T$. 
\end{proof}
\end{appendices}

\bmhead{Acknowledgements}
%
%
Giusy Mazzone is a member of ``Gruppo Nazionale per l'Analisi Matematica, la Probabilit\`a e le loro Applicazioni'' (GNAMPA) of Istituto Nazionale di Alta Matematica Francesco Severi (INdAM).

\section*{Declarations}
%
%
\subsection{Funding}
Giusy Mazzone gratefully acknowledges the support of the Natural Sciences and Engineering Research Council of Canada (NSERC) through the NSERC Discovery Grant ``Partially dissipative systems with applications to fluid-solid interaction problems''.

\subsection{Conflict of interest/Competing interests}
Not applicable. 

\subsection{Ethics approval and consent to participate}
Not applicable. 

\subsection{Consent for publication}
Not applicable.

\subsection{Data availability}
Not applicable. 

\subsection{Materials availability}
Not applicable.

\subsection{Code availability}
Not applicable.


\begin{thebibliography}{39}
\bibitem[{{Mucha, P.~B.} and {Zaj\polhk aczkowski, W.}(2000)}]{mucha00}
{Mucha, P.~B.}, {Zaj\polhk aczkowski, W.} (2000) On the existence for the
  {C}auchy-{N}eumann problem for the {S}tokes system in the {$L_p$}-framework.
  Studia Math. 143(1):75--101.
\end{thebibliography}
\end{document}